\documentclass[11pt,reqno]{amsart}

\usepackage[a4paper,margin=31mm]{geometry}
\usepackage{amsmath,amssymb,amsfonts,amsthm,mathtools,mathrsfs}
\usepackage{microtype}
\usepackage{enumitem}
\usepackage{xurl}
\usepackage{hyperref}
\usepackage[nameinlink,noabbrev,capitalize]{cleveref}
\usepackage[T1]{fontenc}
\usepackage{comment}
\hypersetup{
  colorlinks=true,
  linkcolor=black,
  citecolor=black,
  urlcolor=black,
  pdftitle={Global persistence of nearly radial concentrated vortices},
  pdfauthor={Daomin Cao and Guodong Wang}
}

\numberwithin{equation}{section}
\theoremstyle{plain}
\newtheorem{theorem}{Theorem}[section]
\newtheorem{proposition}[theorem]{Proposition}
\newtheorem{lemma}[theorem]{Lemma}
\newtheorem{corollary}[theorem]{Corollary}
\theoremstyle{remark}
\newtheorem{remark}[theorem]{Remark}
\crefname{theorem}{Theorem}{Theorems}
\crefname{proposition}{Proposition}{Propositions}
\crefname{lemma}{Lemma}{Lemmas}
\crefname{corollary}{Corollary}{Corollaries}
\crefname{remark}{Remark}{Remarks}

\newcommand{\A}{\mathcal A}
\newcommand{\D}{\mathcal D}
\newcommand{\Pcal}{\mathcal P}
\newcommand{\Rcal}{\mathcal R}
\newcommand{\Ecal}{\mathcal E}
\newcommand{\Cset}{\mathsf C}
\newcommand{\bR}{\mathbb R}

\title[Global persistence of nearly radial concentrated vortices]
{Global Persistence of Nearly Radial Concentrated Vortices in a Bounded Domain}

\author{Daomin Cao}
\address{State Key Laboratory of Mathematical Sciences, Academy of Mathematics and Systems Science, Chinese Academy of Sciences, Beijing 100190, P.R. China}
\email{dmcao@amt.ac.cn}

\author{Guodong Wang}
\address{School of Mathematical Sciences, Dalian University of Technology,
Dalian 116024, P.R. China}
\email{gdw@dlut.edu.cn}

\subjclass[2020]{Primary 76B47; Secondary 35Q31, 35B35}
\keywords{2D Euler equation, concentrated vorticity,
symmetric decreasing rearrangement, global persistence,
Kirchhoff--Routh function}

\begin{document}

\begin{abstract}
In this paper, we study the evolution of nearly radial concentrated vortices
for the incompressible Euler equation in a bounded planar domain. We prove
that if a single vortex is initially concentrated near a strict local minimum
point of the Robin function of the domain and is close, up to translation, to
its symmetric decreasing rearrangement, then both its shape and location
remain uniformly controlled for all time. We also establish an analogous
result for a pair of opposite-sign vortices near a strict local minimum point
of the corresponding Kirchhoff--Routh function. No symmetry is imposed on
the domain or the initial data, and the initial data need not be close to any
steady state. To prove these results, we develop a new Lyapunov mechanism for the
evolution of vorticity governed by the Euler equation starting from such initial data,
providing quantitative control of both vortex shape and location. Specifically,
we combine kinetic-energy conservation and vorticity equimeasurability with several
fixed-time estimates to obtain a conditional estimate, and
then use a set-valued first-exit argument to propagate it
globally in time.
\end{abstract}

\maketitle
\setcounter{tocdepth}{1}
\tableofcontents
\section{Introduction and main results}

\subsection{2D Euler dynamics and point vortices}
\label{subsec:euler-point-vortices}
Let $D\subset\bR^2$ be a bounded simply connected domain with a smooth
boundary $\partial D$.  Consider an incompressible inviscid fluid in $D$
subject to the impermeability condition on $\partial D$.  In vorticity
form, the governing Euler equation reads
\begin{equation}\label{eq:euler-intro}
  \begin{cases}
    \partial_t\omega+\mathbf u\cdot\nabla\omega=0,
    &\mathbf x=(x_1,x_2)\in D,\quad t\in\bR,\\
    \mathbf u=\nabla^\perp(-\Delta)^{-1}\omega,\\
    \omega|_{t=0}=\omega_0,
  \end{cases}
\end{equation}
where $\mathbf u$ is the velocity field, $\omega$ is the scalar vorticity,
$\nabla^\perp=(\partial_{x_2},-\partial_{x_1})$, and
$(-\Delta)^{-1}$ denotes the inverse of $-\Delta$ subject to the
zero Dirichlet boundary condition.
For $\omega_0\in L^\infty(D)$, Yudovich's theory \cite{Yudovich1963} gives a unique global weak
solution $\omega\in L^\infty(\bR;L^\infty(D))$ and an associated
measure-preserving flow map $X$ satisfying
\begin{equation*}
  \omega(t,X(t,\mathbf x))=\omega_0(\mathbf x)
  \quad\text{for every }t\in\bR
  \text{ and for a.e. }\mathbf x\in D.
\end{equation*}
See also the detailed accounts in
\cite[Sections~2.2--2.3]{MarchioroPulvirenti1994} and
\cite[Chapter~8]{MajdaBertozzi2002}.

Write $\omega_t:=\omega(t,\cdot)$ and
$\mathbf u_t:=\mathbf u(t,\cdot)$.  The following two conservation laws for
Yudovich solutions of \eqref{eq:euler-intro} will be used throughout this
paper.  The first is \emph{conservation of kinetic energy}:
\begin{equation}\label{def_knc}
  \Ecal(\omega_t)
  :=\frac12\int_D|\mathbf u_t|^2\,\mathrm d\mathbf x
  =\frac12\int_D
  \omega_t(-\Delta)^{-1}\omega_t\,\mathrm d\mathbf x
  =\Ecal(\omega_0).
\end{equation}
The second is \emph{equimeasurability of vorticity}:
\[
  |\{\mathbf x\in D:\omega_t(\mathbf x)>s\}|
  =
  |\{\mathbf x\in D:\omega_0(\mathbf x)>s\}|,
  \qquad (s\in\bR).
\]
In particular, the circulation and every $L^p$ norm,
$1\leq p\leq\infty$, are conserved.  More generally, for any
Borel function $\Psi$ such that $\Psi(\omega_0)\in L^1(D)$,
\begin{equation*}
  \int_D\Psi(\omega_t)\,\mathrm d\mathbf x
  =\int_D\Psi(\omega_0)\,\mathrm d\mathbf x.
\end{equation*}
Detailed proofs of these two conservation laws can be found in
\cite[Section~3]{Burton2005}.

For $N\geq1$, define the configuration space
\begin{equation*}
  \Lambda_N(D)
  :=\{(\mathbf a_1,\ldots,\mathbf a_N)\in D^N:
  \mathbf a_i\neq\mathbf a_j\ \text{for }i\neq j\}.
\end{equation*}
Let G denote the Green function of $-\Delta$ in D with zero Dirichlet boundary condition.  Write
\begin{equation}\label{eq:intro-green}
  G(\mathbf x,\mathbf y)
  =\frac1{2\pi}\ln\frac1{|\mathbf x-\mathbf y|}
  -h(\mathbf x,\mathbf y),\qquad
  (\mathbf x,\mathbf y)\in\Lambda_2(D),
\end{equation}
where $h$ is called the regular part of $G$. Note that $h$ can be
extended to a smooth function on $D\times D$.
In terms of the Green function, the kinetic energy in \eqref{def_knc} has the
following representation:
\begin{equation}\label{def_kngf}
  \Ecal(\omega_t)=\frac12\iint_{D\times D}
  G(\mathbf x,\mathbf y)\omega_t(\mathbf x)\omega_t(\mathbf y)
  \,\mathrm d\mathbf x\,\mathrm d\mathbf y.
\end{equation}
The \emph{Robin function} of $D$ is defined by
\[
  H(\mathbf x):=\frac12 h(\mathbf x,\mathbf x).
\]
Given nonzero circulations
$\boldsymbol\kappa=(\kappa_1,\ldots,\kappa_N)$,
the corresponding \emph{Kirchhoff--Routh function}
$W_{\boldsymbol\kappa}^{(N)}:\Lambda_N(D)\to\bR$
is defined by
\begin{equation}\label{eq:intro-KR}
  W_{\boldsymbol\kappa}^{(N)}
  (\mathbf a_1,\ldots,\mathbf a_N)
  :=\sum_{i=1}^N\kappa_i^2H(\mathbf a_i)
  -\sum_{1\leq i<j\leq N}
  \kappa_i\kappa_j\,G(\mathbf a_i,\mathbf a_j).
\end{equation}
Thus, $W_{\boldsymbol\kappa}^{(N)}$ is a circulation-weighted combination
of the Robin function $H$ and the Green function $G$.

We next introduce the \emph{point-vortex model}, which formally arises as the
singular limit of concentrated Euler vorticity.  Replacing a distributed
vorticity field by
$\sum_{i=1}^N\kappa_i\delta_{\mathbf a_i(t)}$ and removing the singular
self-interaction of each vortex formally yields the following Hamiltonian
system for the point-vortex positions
$\mathbf a_i(t)$:
\begin{equation*}
  \kappa_i\frac{\mathrm d{\mathbf a}_i}{\mathrm dt}
  =-\nabla_{\mathbf a_i}^{\perp}
  W_{\boldsymbol\kappa}^{(N)}
  (\mathbf a_1,\ldots,\mathbf a_N),
  \qquad i=1,\ldots,N.
\end{equation*}
See \cite[Chapter~4]{MarchioroPulvirenti1994}.  As usual for a Hamiltonian
system, $W_{\boldsymbol\kappa}^{(N)}$ is conserved along this motion.
Hence every strict local extremum of
$W_{\boldsymbol\kappa}^{(N)}$ is a Lyapunov-stable point-vortex
equilibrium.  For one vortex ($N=1$),
$W_\kappa^{(1)}=\kappa^2H$, so a critical point of the Robin
function is a stationary point-vortex location and is Lyapunov-stable
whenever it is a strict local extremum.

The rigorous connection between concentrated Euler vorticity and point-vortex
dynamics was initiated in \cite{MarchioroPulvirenti1983} and subsequently
developed in \cite{MarchioroPagani1986,Turkington1987,Marchioro1988,
MarchioroPulvirenti1993}.  Later work obtained sharper localization estimates and longer
confinement times \cite{Marchioro1998Localization,CapriniMarchioro2015,
CeciSeis2021,CeciSeis2022,ButtaMarchioro2018,DonatiIftimie2021,Guo2025,
Meyer2025}, as well as global-in-time results for several special configurations
\cite{ChoiJeongYao2024,DavilaDelPinoMussoParmeshwar2026,AbeJeongYao2025}.   A detailed discussion of these results can be found in \cref{sec13}.

A natural question is whether concentrated vortices remain concentrated for
\emph{all} time in a \emph{general} bounded domain.  Compared with its
whole-plane counterpart, the bounded-domain problem presents additional
difficulties: the boundary complicates the vortex dynamics, and some
conservation laws are no longer available.  In general, global-in-time
concentration cannot be expected for arbitrary initial vorticity, so
additional assumptions are needed.

In this paper, we identify a natural \emph{point-vortex trapping condition} that
provides the required structure.  For a single vortex, we assume that the
initial vorticity is nearly radial and concentrated near a strict
local minimum point of the Robin function.  For a pair of
opposite-sign vortices, each
component is nearly radial and the two initial centers lie near a strict local
minimum point of the corresponding two-vortex Kirchhoff--Routh function.  Under
these assumptions, we develop a \emph{new Lyapunov mechanism} that provides uniform-in-time control of vortex shape and of the vortex-center locations, thereby proving global-in-time persistence of concentration.

\subsection{Main results}

For a nonnegative function $f\in L^1(\bR^2)$, let $f^*$ denote its
\emph{symmetric decreasing rearrangement} (see \cite[Chapter~3]{LiebLoss2001}) and define
\begin{equation}\label{def_oaf}
  \A(f):=\inf_{\mathbf b\in\bR^2}
  \|f(\cdot+\mathbf b)-f^*\|_1,
  \qquad
  \Cset(f):=\{\mathbf b\in\bR^2:
  \|f(\cdot+\mathbf b)-f^*\|_1=\A(f)\}.
\end{equation}
Here and below, $\|\cdot\|_p$ denotes the $L^p(\bR^2)$ norm for
$1\leq p\leq\infty$. The functional $\A(f)$ measures the
\emph{orbital asymmetry} of $f$: it is the distance, modulo translations,
from $f$ to its symmetric decreasing rearrangement. The set $\Cset(f)$
contains all translation centers that realize this distance, and its elements
are called \emph{optimal translation centers}. The properties of
$\Cset(f)$ needed below are established in \cref{app:centers}.

From now on, whenever a function on $D$ appears in a full-plane expression,
it is understood to be extended by zero to $\bR^2$. By
$\operatorname{supp}f$, we always mean the \emph{essential support} of $f$; see
\cite[Section~1.1]{LiebLoss2001}.

We are now ready to state our global-in-time persistence results. For clarity
and brevity, we give only their qualitative forms here; the corresponding
quantitative estimates are stated in \cref{sec:setting}.

\begin{theorem}[Global persistence: one concentrated vortex]
\label{thm:intro-one}
Let $\mathbf a_*\in D$ be a strict local minimum point of $H$.  Fix
$\kappa,M,R>0$.  For each $\varepsilon>0$, let
$\mathbf a_\varepsilon\in D$ and let
$\omega_{\varepsilon,0}\in L^\infty(D)$ be nonnegative and satisfy
\begin{equation}\label{inicond1}
  \int_D\omega_{\varepsilon,0}\,\mathrm d\mathbf x=\kappa,
  \qquad
  \|\omega_{\varepsilon,0}\|_\infty\leq M\varepsilon^{-2},
  \qquad
  \operatorname{supp}\omega_{\varepsilon,0}
  \subset B_{R\varepsilon}(\mathbf a_\varepsilon).
\end{equation}
Suppose that
\begin{equation*}
  \mathbf a_\varepsilon\longrightarrow\mathbf a_*,
  \qquad
  \A(\omega_{\varepsilon,0})\longrightarrow0
  \qquad\text{as }\varepsilon\longrightarrow0.
\end{equation*}
Let $\omega_{\varepsilon,t}$, $t\in\bR$, denote the corresponding
Yudovich solution of \eqref{eq:euler-intro}.
Then, as $\varepsilon\to0$,
\begin{equation}\label{eq:qualitative-conclusion1}
  \sup_{t\in\bR}\A(\omega_{\varepsilon,t})\longrightarrow0,
\end{equation}
and
\begin{equation}\label{eq:qualitative-conclusion2}
  \sup_{t\in\bR}\ \sup_{\mathbf b\in
  \Cset(\omega_{\varepsilon,t})}
  |\mathbf b-\mathbf a_*|\longrightarrow0.
\end{equation}
\end{theorem}

\begin{remark}\label{remk12}
By \cite[Lemma~2.2]{Turkington1987},
$H(\mathbf x)\to+\infty$ as
$\mathbf x\to\partial D$. Hence $H$ has at least one global minimum
point in $D$. If $D$ is additionally convex, Caffarelli and Friedman
\cite[Theorem~3.1]{CaffarelliFriedman1985} proved that $H$ is strictly
convex. Its global minimum point is therefore unique and hence strict.
More generally, Bartsch, Micheletti, and Pistoia
\cite[Theorem~1.1]{BartschMichelettiPistoia2019} proved that the Robin and
Kirchhoff--Routh functions are Morse functions for a residual set of sufficiently
small smooth perturbations of any fixed bounded smooth domain. For each
such perturbed domain, every global minimum point of $H$ is nondegenerate
and hence strict. Thus, \cref{thm:intro-one} applies to every bounded smooth
convex domain and, in the above generic sense, to ``most'' bounded smooth
simply connected domains.
\end{remark}

\begin{remark}
The conditions in \eqref{inicond1} describe an initial vortex concentrated on
the spatial scale $O(\varepsilon)$.  Such assumptions, or closely related
variants, are standard in the vortex-blob approximation and
vorticity-confinement literature; see, for example,
\cite{MarchioroPulvirenti1983,MarchioroPagani1986,Marchioro1988,
ButtaMarchioro2018,DonatiIftimie2021,Meyer2025}.
\end{remark}

\begin{remark}\label{rk140}
By \cref{lem:centers} in \cref{app:centers}, the set of optimal translation centers
$\Cset(\omega_{\varepsilon,t})$ is nonempty and compact for every
$t\in\bR$. Moreover, every
$\mathbf b\in\Cset(\omega_{\varepsilon,t})$ satisfies
\[
  \int_{D\setminus B_{R\varepsilon}(\mathbf b)}
  \omega_{\varepsilon,t}\,\mathrm d\mathbf x
  \leq \frac12\A(\omega_{\varepsilon,t}).
\]
Combining this inequality with
\eqref{eq:qualitative-conclusion1}, we find
\begin{equation*}
  \sup_{t\in\bR}\ \sup_{\mathbf b\in
  \Cset(\omega_{\varepsilon,t})}
  \int_{D\setminus B_{R\varepsilon}(\mathbf b)}
  \omega_{\varepsilon,t}\,\mathrm d\mathbf x
  \longrightarrow0
\end{equation*}
as $\varepsilon\to0$.
That is, all but $o(1)$ of the circulation
lies in an $O(\varepsilon)$-radius ball centered at each optimal center, uniformly in time.
\end{remark}

Our second result concerns two concentrated vortices of opposite signs near
a strict local minimum point of the corresponding Kirchhoff--Routh function.

\begin{theorem}[Global persistence: two opposite-sign vortices]
\label{thm:intro-two}
Let $\boldsymbol\kappa=(\kappa_1,\kappa_2)$ satisfy
$\kappa_1>0>\kappa_2$, and let
$(\mathbf a_{*,1},\mathbf a_{*,2})\in\Lambda_2(D)$
be a strict local minimum point of the corresponding Kirchhoff--Routh function
$W_{\boldsymbol\kappa}^{(2)}$ defined in \eqref{eq:intro-KR}.  Fix
$M_i,R_i>0$ for $i=1,2$.  For each $\varepsilon>0$, let
$\mathbf a_{1,\varepsilon},\mathbf a_{2,\varepsilon}\in D$, and let
$\zeta_{i,\varepsilon,0}\in L^\infty(D)$ be nonnegative and satisfy
\[
  \int_D\zeta_{i,\varepsilon,0}\,\mathrm d\mathbf x=|\kappa_i|,
  \qquad
  \|\zeta_{i,\varepsilon,0}\|_\infty\leq M_i\varepsilon^{-2},
  \qquad
  \operatorname{supp}\zeta_{i,\varepsilon,0}
  \subset B_{R_i\varepsilon}(\mathbf a_{i,\varepsilon}),
  \qquad i=1,2.
 \]
Suppose that
\[
  (\mathbf a_{1,\varepsilon},\mathbf a_{2,\varepsilon})
  \longrightarrow
  (\mathbf a_{*,1},\mathbf a_{*,2}),
  \qquad
  \A(\zeta_{i,\varepsilon,0})\longrightarrow0
  \quad(i=1,2)
\]
as $\varepsilon\to0$.  Set
$\omega_{\varepsilon,0}
:=\zeta_{1,\varepsilon,0}-\zeta_{2,\varepsilon,0}$.
Let $\omega_{\varepsilon,t}$, $t\in\bR$, be the corresponding
Yudovich solution of \eqref{eq:euler-intro}, and, for $i=1,2$, let
$\zeta_{i,\varepsilon,t}$ denote the transport of
$\zeta_{i,\varepsilon,0}$ by the associated Euler flow map
(so that
$\omega_{\varepsilon,t}
=\zeta_{1,\varepsilon,t}-\zeta_{2,\varepsilon,t}$
for every $t\in\bR$; see \cref{rk:decomvor}).
Then, as $\varepsilon\to0$,
\begin{equation}\label{eq:two-qualitative}
  \sup_{t\in\bR}\A(\zeta_{i,\varepsilon,t})
  \longrightarrow0
  \qquad(i=1,2),
\end{equation}
and
\begin{equation}\label{eq:two-center-limit}
  \sup_{t\in\bR}\ \sup_{(\mathbf b_1,\mathbf b_2)\in
  \Cset(\zeta_{1,\varepsilon,t})\times
  \Cset(\zeta_{2,\varepsilon,t})}
  |(\mathbf b_1,\mathbf b_2)
  -(\mathbf a_{*,1},\mathbf a_{*,2})|
  \longrightarrow0.
\end{equation}
\end{theorem}

\begin{remark}\label{rk:decomvor}
Let
$\mathbf u_{\varepsilon,t}
:=\nabla^\perp(-\Delta)^{-1}\omega_{\varepsilon,t}$.
Then
\[
  \partial_t\omega_{\varepsilon,t}
  +\mathbf u_{\varepsilon,t}\cdot\nabla\omega_{\varepsilon,t}=0,
  \qquad
  \omega_{\varepsilon,0}
  =\zeta_{1,\varepsilon,0}-\zeta_{2,\varepsilon,0}.
\]
Since each $\zeta_{i,\varepsilon,t}$ is transported by the same velocity,
linearity gives
\[
  \partial_t(\zeta_{1,\varepsilon,t}-\zeta_{2,\varepsilon,t})
  +\mathbf u_{\varepsilon,t}\cdot\nabla
  (\zeta_{1,\varepsilon,t}-\zeta_{2,\varepsilon,t})=0,
\]
with initial value
$\zeta_{1,\varepsilon,0}-\zeta_{2,\varepsilon,0}
=\omega_{\varepsilon,0}$.
Hence uniqueness for the linear transport equation
\cite[Lemma~12(ii)]{Burton2005} gives
\[
  \omega_{\varepsilon,t}
  =\zeta_{1,\varepsilon,t}-\zeta_{2,\varepsilon,t}
  \qquad(t\in\bR).
\]
\end{remark}

\begin{remark}
When $\kappa_1\kappa_2<0$, the boundary asymptotics of the Green and Robin
functions imply that
\[
  W_{\boldsymbol\kappa}^{(2)}(\mathbf a_1,\mathbf a_2)
  \longrightarrow+\infty
  \qquad\text{as }
  (\mathbf a_1,\mathbf a_2)\to\partial\Lambda_2(D).
\]
It follows that $W_{\boldsymbol\kappa}^{(2)}$ has at least one global minimum
point in $\Lambda_2(D)$.
By the generic Morse property recalled in \cref{remk12}, for ``most''
bounded smooth simply connected domains, every global minimum point is
nondegenerate and hence strict. In \cref{prop:two-example}, we also provide
an explicit family of analytic domains, obtained by perturbing the disk, for which
$W_{\boldsymbol\kappa}^{(2)}$ with $\kappa_1=-\kappa_2$ has a strict
nondegenerate local minimum point.
\end{remark}

\begin{remark}
In the setting of \cref{thm:intro-two}, the argument in \cref{rk140}
applies separately to each vortex component and yields
\begin{equation*}
  \sup_{t\in\bR}\ \sup_{\mathbf b_i\in
  \Cset(\zeta_{i,\varepsilon,t})}
  \int_{D\setminus B_{R_i\varepsilon}(\mathbf b_i)}
  \zeta_{i,\varepsilon,t}\,\mathrm d\mathbf x
  \longrightarrow0
  \qquad(i=1,2)
\end{equation*}
as $\varepsilon\to0$.
In other words, for each $i=1,2$, all but $o(1)$ of the circulation of the
$i$th vortex lies in an $O(\varepsilon)$-radius ball centered at any $\mathbf b_i\in\Cset(\zeta_{i,\varepsilon,t})$,
uniformly in time.
\end{remark}

We emphasize that \cref{thm:intro-one,thm:intro-two} only show that most of the circulation remains localized; they do not control the diameter of the full support. Thin filaments may form, but they carry only a vanishing amount of circulation.

We also note that the assumption that $D$ is simply connected does not appear to be essential to our approach. We expect that analogous results can be established in bounded multiply connected domains, provided that the circulations around the inner boundary components are fixed and that the Kirchhoff--Routh function is replaced by the corresponding generalized Hamiltonian (see (1.4) in \cite{Donati2023}).

\subsection{Relation to earlier work}\label{sec13}

As reviewed in \cref{subsec:euler-point-vortices}, the evolution of
concentrated vortices has been studied extensively. Our work differs from earlier work in three main respects. First, the control is global in
time, whereas most previous approximation and confinement results apply only
on finite time intervals.  Second, within the point-vortex trapping regime,
the initial data need only be concentrated and nearly radial; they need not
be close to a steady Euler flow.  Third, we allow a general bounded domain and impose no symmetry on either the domain or the initial data.
Below we discuss these distinctions in more detail by comparing our results
with three closely related lines of work.

\begin{enumerate}[label=\textup{(\roman*)}]

\item  \emph{Point-vortex approximation and long-time confinement.}
The literature on point-vortex approximation has developed from fixed-time
convergence results to longer, scale-dependent confinement estimates.  The
classical vortex-concentration results
\cite{MarchioroPulvirenti1983,MarchioroPagani1986,Marchioro1988,
MarchioroPulvirenti1993,Turkington1987}, together with later quantitative
refinements \cite{CeciSeis2021}, show that concentrated Euler vorticity
follows point-vortex dynamics on each fixed time interval on which the
limiting vortices remain mutually separated and away from the boundary.
Butt\`a and Marchioro \cite{ButtaMarchioro2018} obtain longer,
$\varepsilon$-dependent confinement time scales under symmetry assumptions,
while Donati and Iftimie \cite{DonatiIftimie2021} prove algebraically long
confinement of the material support around a stationary point vortex under
an additional domain-dependent condition.  Donati \cite{Donati2025} obtains
algebraically long support confinement near special polygonal vortex
crystals.  Meyer \cite{Meyer2025} allows
multiple separated, nearly circular vortices over long time scales and proves
a stability estimate for the logarithmic interaction energy.

The spatial and configurational scope of these results differs from ours.
Donati \cite{Donati2025} and Donati and Iftimie
\cite{DonatiIftimie2021} control the full support, whereas we control the circulation outside a small ball.  The confinement intervals in those works remain finite, though long, and rely on special point-vortex configurations or an additional condition
on the domain.  Meyer \cite{Meyer2025} allows multiple vortices, whereas here we restrict attention to one vortex or one opposite-sign pair and allow the relevant strict local minimum point to be degenerate.

The decisive distinction, however, is the time horizon.  In all the
approximation and confinement results cited in this item, one either
prescribes an arbitrary finite time interval before taking the concentration
limit, or obtains an $\varepsilon$-dependent interval that, although long,
remains finite for each fixed $\varepsilon$.  Thus none of these results
provides global-in-time control at a fixed concentration scale.  Our results
do: for every sufficiently small fixed $\varepsilon$, the control is uniform
over all $t\in\bR$.

\item  \emph{Nonlinear stability of steady concentrated Euler flows.}
Variational methods have been used to construct steady Euler flows with
concentrated vorticity and establish their nonlinear stability in bounded
domains; see \cite{Burton2005,CaoWang2019,CaoWang2021,Wang2024}.  In
particular, Cao and Wang \cite{CaoWang2019,CaoWang2021} studied the existence
and nonlinear stability of steady vortex patches concentrated near a strict
local minimum point of the Robin function or the opposite-sign two-vortex
Kirchhoff--Routh function. Wang \cite{Wang2024} subsequently extended these
results to steady Euler flows with more general vorticity distributions.

By contrast, our results do not require the initial data to be close to any
steady Euler flow. Subject to the concentration and point-vortex trapping
assumptions in \cref{thm:intro-one,thm:intro-two}, the initial vorticity
profiles may otherwise be arbitrary, provided that their orbital asymmetry
is small. Thus, our results do not follow from the nonlinear stability of
any particular steady flow.

\item  \emph{Global-in-time results for special configurations.}
Choi, Jeong, and Yao \cite{ChoiJeongYao2024} prove global-in-time
stability for planar vortex quadrupoles that are odd with respect to both coordinate axes.  As in our setting, their initial data consist of
concentrated vortices that are close, up to translation, to their radial
rearrangements.  Their theorem gives global-in-time control near an
odd--odd point-vortex orbit.  Among the available global-in-time stability
results for interacting concentrated vortices, this is especially close to
our setting because it does not assume closeness to an equilibrium or a
relative equilibrium.  For comparison, a simpler whole-plane single-vortex
case is recorded in \cref{thm:whole-plane-persistence}.  The main distinction
between their result and ours nevertheless lies not in the local vortex
profiles but in the odd--odd symmetry imposed on the whole-plane
configuration.

Other global-in-time results are known for expanding or specially constructed
periodic configurations.  Zbarsky \cite{Zbarsky2021} treats three vortex
blobs near a self-similarly expanding point-vortex configuration in the whole
plane.  Hassainia, Hmidi, and Masmoudi \cite{HassainiaHmidiMasmoudi2025}
construct an all-time leapfrogging quartet of concentrated vortex patches,
while Hassainia, Hmidi, and Roulley \cite{HassainiaHmidiRoulley2026}
construct time-periodic concentrated vortex patches in bounded simply connected
domains under nondegeneracy assumptions.  Related whole-plane results
also include the four-vortex solutions constructed by D\'avila, del Pino,
Musso, and Parmeshwar \cite{DavilaDelPinoMussoParmeshwar2026} and the
stability of multiple Lamb dipoles proved by Abe, Jeong, and Yao
\cite{AbeJeongYao2025}.  These results concern specific expanding,
periodic, or traveling structures, or perturbations of Lamb dipoles; none
treats arbitrary nearly radial concentrated profiles trapped near a
point-vortex equilibrium in a general bounded domain.

The Choi--Jeong--Yao position estimate is formulated in terms of the mass
center of the vortex in a quadrant, whereas our argument uses the full set of
optimal translation centers.  The reasons for this difference are discussed
in \cref{subsec:mass-centers}.
\end{enumerate}

Our approach may also apply to other concentrated vortex solutions with a
similar energy-based Lyapunov mechanism.  Examples include traveling and
rotating vortex pairs in the plane, and vortex pairs of equal and opposite
strengths on the flat torus or the sphere.  In each case, one would use the
appropriate Green function and an energy
functional adapted to the relevant moving frame.

\subsection{Organization of the paper}

In \cref{sec:setting}, we introduce the notation, state the quantitative
theorems, and outline the proof strategy.  The three fixed-time estimates are
then established in order: the quantitative rearrangement estimates in
\cref{sec:quantitative-rearrangement}, the logarithmic tail estimate in
\cref{sec:estimates_tail}, and the regular energy estimate in
\cref{sec:robin}.  We prove the one-vortex result in \cref{sec:global} and the
opposite-sign two-vortex result in \cref{sec:two-vortices}.
\Cref{sec:further-discussion} contains three further discussions: why the
Lyapunov mechanism does not extend directly to the same-sign case, why the
$L^1$ norm is the natural orbital distance, and why the present argument uses
optimal translation centers rather than mass centers.  The appendices collect
several auxiliary results, including the basic properties of optimal centers,
a set-valued first-exit lemma, and the construction of an explicit
noncircular domain that admits a trapped opposite-sign pair.

\section{Quantitative estimates and proof strategy}
\label{sec:setting}

\subsection{Several functionals}
In this subsection, we introduce several functionals used throughout the paper.

For a compactly supported $f\in L^\infty(\bR^2)$, define its
\emph{logarithmic self-interaction energy} by
\begin{equation*}
  \Pcal(f):=\frac1{4\pi}
  \iint_{\bR^2\times\bR^2}
  \ln\frac1{|\mathbf x-\mathbf y|}
  f(\mathbf x)f(\mathbf y)\,\mathrm d\mathbf x\,\mathrm d\mathbf y.
\end{equation*}
For every $1<p<\infty$, there exists $C_{D,p}>0$ such that, for all
$f,g\in L^\infty(D)$,
\begin{equation}\label{eq:P-Lipschitz}
  |\Pcal(f)-\Pcal(g)|
  \leq C_{D,p}(\|f\|_p+\|g\|_p)\|f-g\|_p.
\end{equation}
This follows from H\"older's inequality, since, for
$p'=p/(p-1)$, the logarithmic kernel belongs to
$L^{p'}(D\times D)$.

For a nonnegative, compactly supported $f\in L^\infty(\bR^2)$, define its
\emph{logarithmic rearrangement deficit} by
\begin{equation*}
  \D(f):=\Pcal(f^*)-\Pcal(f).
\end{equation*}
The Riesz rearrangement inequality \cite[Theorem~3.7]{LiebLoss2001} gives
$\D(f)\geq0$.  The scale-uniform estimates relating $\D$ and the orbital asymmetry functional $\A$ (see \eqref{def_oaf}) are given in \cref{sec:quantitative-rearrangement}.

By \eqref{eq:intro-green} and \eqref{def_kngf}, for
$f\in L^\infty(D)$, the kinetic energy has the decomposition
\begin{equation}\label{eq:energy-split}
  \Ecal(f)=\Pcal(f)-\Rcal(f),\qquad \Rcal(f):=\frac12\iint_{D\times D}
  h(\mathbf x,\mathbf y)f(\mathbf x)f(\mathbf y)
  \,\mathrm d\mathbf x\,\mathrm d\mathbf y,
\end{equation}
where $\Rcal$ is called the \emph{regular energy}.  For
$f,g\in L^\infty(D)$, we also use the \emph{cross interaction}
\begin{equation}\label{eq:cross-interaction}
  \mathcal I(f,g):=
  \iint_{D\times D}G(\mathbf x,\mathbf y)
  f(\mathbf x)g(\mathbf y)
  \,\mathrm d\mathbf x\,\mathrm d\mathbf y.
\end{equation}

\subsection{Quantitative estimates}

In this subsection, we state quantitative estimates for the shape and
location of concentrated vortices.  We measure the shape error by $\D$ and
the position error by the corresponding Robin or Kirchhoff--Routh gap.
Although less direct than $\A$ and the Euclidean distance from the reference
point-vortex configuration, these quantities are naturally connected to
energy conservation and equimeasurability.  Together with the scale-uniform comparisons between $\D$ and $\A$
established in \cref{sec:quantitative-rearrangement}, these estimates readily yield the qualitative results
\cref{thm:intro-one,thm:intro-two}.

\subsubsection{One concentrated vortex}
For sets $A,B\subset\bR^2$, we write $A\Subset B$ if
$\overline A$ is a compact subset of $B$.

\begin{theorem}[Quantitative estimates: one vortex]
\label{thm:main}
Let $\mathbf a_*\in D$ be a strict local minimum point of $H$, and choose
$\rho>0$ such that
\begin{equation}\label{eq:well}
  B_\rho(\mathbf a_*)\Subset D,
\end{equation}
and
\begin{equation}\label{eq:strict-well}
  H(\mathbf b)>H(\mathbf a_*)
  \quad
  \bigl(
  \mathbf b\in\overline{B_\rho(\mathbf a_*)}
  \setminus\{\mathbf a_*\}
  \bigr).
\end{equation}
Fix $\kappa,M,R>0$.  Then there exist positive constants
$C$, $\eta_0$, and $\varepsilon_0$, depending only on
$D,\mathbf a_*,\rho,\kappa,M$, and $R$, such that the following holds.

For every $0<\varepsilon<\varepsilon_0$, let
$\mathbf a_\varepsilon\in B_\rho(\mathbf a_*)$, and let
$\omega_{\varepsilon,0}\in L^\infty(D)$ be nonnegative and satisfy
\[
  \int_D\omega_{\varepsilon,0}\,\mathrm d\mathbf x=\kappa,
  \qquad
  \|\omega_{\varepsilon,0}\|_\infty\leq M\varepsilon^{-2},
  \qquad
  \operatorname{supp}\omega_{\varepsilon,0}
  \subset B_{R\varepsilon}(\mathbf a_\varepsilon).
\]
Let $\omega_{\varepsilon,t}$, $t\in\bR$, be the corresponding Yudovich
solution of \eqref{eq:euler-intro}.  If
\begin{equation}\label{eq:quantitative-smallness}
  \eta_{\varepsilon,0}
  :=\D(\omega_{\varepsilon,0})
  +\kappa^2[H(\mathbf a_\varepsilon)-H(\mathbf a_*)]
  +\varepsilon
  \leq\eta_0,
\end{equation}
then, for every $t\in\bR$,
\begin{equation}\label{eq:center-confinement}
  \Cset(\omega_{\varepsilon,t})\subset B_\rho(\mathbf a_*)
\end{equation}
and
\begin{equation}\label{eq:quantitative-main}
  \D(\omega_{\varepsilon,t})
  +\kappa^2\sup_{\mathbf b\in\Cset(\omega_{\varepsilon,t})}
  [H(\mathbf b)-H(\mathbf a_*)]
  \leq C\eta_{\varepsilon,0}.
\end{equation}
\end{theorem}

\begin{remark}
The $\varepsilon$ contribution in $\eta_{\varepsilon,0}$, and hence on the
right-hand side of \eqref{eq:quantitative-main}, comes from
approximating the regular energy of a concentrated vortex by
the Robin energy of a point vortex; see
\cref{prop:robin-lower,lem:initial-robin}.  In general, this term cannot be
removed.  Indeed, for centered radial initial data
as in the following corollary,
all the terms in $\eta_{\varepsilon,0}$ other than $\varepsilon$ vanish.
Without the $\varepsilon$ term, the
estimate would force
\[
  \D(\omega_{\varepsilon,t})=0,
  \qquad
  \Cset(\omega_{\varepsilon,t})=\{\mathbf a_*\}
  \quad(t\in\bR).
\]
Thus every such initial vortex would be a steady solution, which is generally
false.
\end{remark}

The following is an immediate consequence of \cref{thm:main}.

\begin{corollary}[Centered radial initial data]\label{coro_crdata}
Under the assumptions of \cref{thm:main}, suppose in addition that
\[
  \mathbf a_\varepsilon=\mathbf a_*,
  \qquad
  \omega_{\varepsilon,0}(\mathbf x)
  =\omega_{\varepsilon,0}^*(\mathbf x-\mathbf a_*).
\]
Then, for all sufficiently small $\varepsilon>0$,
\begin{equation}\label{cont_db}
  \sup_{t\in\bR}\left\{
  \D(\omega_{\varepsilon,t})
  +\kappa^2\sup_{\mathbf b\in\Cset(\omega_{\varepsilon,t})}
  [H(\mathbf b)-H(\mathbf a_*)]\right\}
  \leq C\varepsilon,
\end{equation}
and
\begin{equation}\label{eq:natural-shape-scale}
  \sup_{t\in\bR}\A(\omega_{\varepsilon,t})
  \leq C\varepsilon^{1/2}.
\end{equation}
If $D^2H(\mathbf a_*)$ is positive definite, then, after possibly decreasing
$\rho$,
\begin{equation}\label{eq:natural-center-scale}
  \sup_{t\in\bR}\ \sup_{\mathbf b\in\Cset(\omega_{\varepsilon,t})}
  |\mathbf b-\mathbf a_*|
  \leq C\varepsilon^{1/2}.
\end{equation}
Here $C>0$ depends only on
$D,\mathbf a_*,\rho,\kappa,M$, and $R$.
\end{corollary}

\begin{proof}
The centered radial assumption gives
$\eta_{\varepsilon,0}=\varepsilon$, so the first estimate follows from
\eqref{eq:quantitative-main}.  By equimeasurability and
\cref{lem:shape-coercivity},
\[
  \A(\omega_{\varepsilon,t})^2
  \leq C\D(\omega_{\varepsilon,t})\leq C\varepsilon,
\]
uniformly in $t$, which proves \eqref{eq:natural-shape-scale}.

If $D^2H(\mathbf a_*)$ is positive definite, then, after decreasing $\rho$ if
necessary,
\[
  H(\mathbf b)-H(\mathbf a_*)\geq c|\mathbf b-\mathbf a_*|^2,
  \qquad \mathbf b\in B_\rho(\mathbf a_*).
\]
Combining this with \eqref{eq:center-confinement} and \eqref{cont_db} gives
\eqref{eq:natural-center-scale}.
\end{proof}

\subsubsection{Two opposite-sign concentrated vortices}

\begin{theorem}[Quantitative estimates: two opposite-sign vortices]
\label{thm:two-vortex}
Let $\boldsymbol\kappa=(\kappa_1,\kappa_2)$ satisfy
$\kappa_1>0>\kappa_2$, and let
$(\mathbf a_{*,1},\mathbf a_{*,2})\in\Lambda_2(D)$ be a strict local
minimum point of $W_{\boldsymbol\kappa}^{(2)}$.  Choose $\rho>0$ such that
\[
  B_\rho(\mathbf a_{*,1})\times B_\rho(\mathbf a_{*,2})
  \Subset\Lambda_2(D),
\]
and
\begin{equation}\label{eq:two-strict-well}
  \begin{aligned}
  W_{\boldsymbol\kappa}^{(2)}(\mathbf b_1,\mathbf b_2)
  &>W_{\boldsymbol\kappa}^{(2)}
  (\mathbf a_{*,1},\mathbf a_{*,2})\\
  &\quad\text{for every }(\mathbf b_1,\mathbf b_2)\in
  \overline{B_\rho(\mathbf a_{*,1})\times
  B_\rho(\mathbf a_{*,2})}
  \setminus\{(\mathbf a_{*,1},\mathbf a_{*,2})\}.
  \end{aligned}
\end{equation}
Fix $M_i,R_i>0$ for $i=1,2$.  Then there exist positive constants
$C$, $\eta_0$, and $\varepsilon_0$, depending only on
$D,\rho$, and $\mathbf a_{*,i},\kappa_i,M_i,R_i$, $i=1,2$,
such that the following holds.

For every $0<\varepsilon<\varepsilon_0$, let
$(\mathbf a_{1,\varepsilon},\mathbf a_{2,\varepsilon})
\in B_\rho(\mathbf a_{*,1})\times B_\rho(\mathbf a_{*,2})$, and let
$\zeta_{i,\varepsilon,0}\in L^\infty(D)$ be nonnegative and satisfy
\[
  \int_D\zeta_{i,\varepsilon,0}\,\mathrm d\mathbf x=|\kappa_i|,
  \qquad
  \|\zeta_{i,\varepsilon,0}\|_\infty
  \leq M_i\varepsilon^{-2},
  \qquad
  \operatorname{supp}\zeta_{i,\varepsilon,0}
  \subset B_{R_i\varepsilon}(\mathbf a_{i,\varepsilon}),
  \quad i=1,2.
\]
Set
$\omega_{\varepsilon,0}
:=\zeta_{1,\varepsilon,0}-\zeta_{2,\varepsilon,0}$.
Let $\omega_{\varepsilon,t}$, $t\in\bR$, be the corresponding Yudovich
solution of \eqref{eq:euler-intro}, and let
$\zeta_{i,\varepsilon,t}$ denote the transport of
$\zeta_{i,\varepsilon,0}$ by the associated Euler flow map.
If
\begin{equation}\label{eq:two-quantitative-smallness}
  \eta^{(2)}_{\varepsilon,0}
  :=\sum_{i=1}^2\D(\zeta_{i,\varepsilon,0})
  +\Bigl[W_{\boldsymbol\kappa}^{(2)}
  (\mathbf a_{1,\varepsilon},\mathbf a_{2,\varepsilon})
  -W_{\boldsymbol\kappa}^{(2)}
  (\mathbf a_{*,1},\mathbf a_{*,2})\Bigr]
  +\varepsilon
  \leq\eta_0,
\end{equation}
then, for every $t\in\bR$,
\begin{equation}\label{eq:two-center-confinement}
  \Cset(\zeta_{1,\varepsilon,t})
  \times\Cset(\zeta_{2,\varepsilon,t})
  \subset B_\rho(\mathbf a_{*,1})\times
  B_\rho(\mathbf a_{*,2}),
\end{equation}
and
\begin{align}
  &\sum_{i=1}^2\D(\zeta_{i,\varepsilon,t})
  +\sup_{(\mathbf b_1,\mathbf b_2)\in
  \Cset(\zeta_{1,\varepsilon,t})\times
  \Cset(\zeta_{2,\varepsilon,t})}
  \Bigl[
  W_{\boldsymbol\kappa}^{(2)}(\mathbf b_1,\mathbf b_2)
  -W_{\boldsymbol\kappa}^{(2)}
  (\mathbf a_{*,1},\mathbf a_{*,2})
  \Bigr] \leq C\eta^{(2)}_{\varepsilon,0}.
  \label{eq:two-quantitative}
\end{align}
\end{theorem}

The corresponding consequence for two centered radial vortices is as follows.

\begin{corollary}[Two centered radial initial vortices]
Under the assumptions of \cref{thm:two-vortex}, suppose in addition that, for
$i=1,2$,
\[
  \mathbf a_{i,\varepsilon}=\mathbf a_{*,i},
  \qquad
  \zeta_{i,\varepsilon,0}(\mathbf x)
  =\zeta_{i,\varepsilon,0}^*
  (\mathbf x-\mathbf a_{*,i}).
\]
Then, for all sufficiently small $\varepsilon>0$,
\begin{equation*}
  \sup_{t\in\bR}\left\{
  \sum_{i=1}^2\D(\zeta_{i,\varepsilon,t})
  +\sup_{(\mathbf b_1,\mathbf b_2)\in
  \Cset(\zeta_{1,\varepsilon,t})\times
  \Cset(\zeta_{2,\varepsilon,t})}
  \Bigl[W_{\boldsymbol\kappa}^{(2)}(\mathbf b_1,\mathbf b_2)
  -W_{\boldsymbol\kappa}^{(2)}
  (\mathbf a_{*,1},\mathbf a_{*,2})\Bigr]\right\}
  \leq C\varepsilon,
\end{equation*}
and
\begin{equation*}
  \sup_{t\in\bR}\sum_{i=1}^2
  \A(\zeta_{i,\varepsilon,t})
  \leq C\varepsilon^{1/2}.
\end{equation*}
If the Hessian of $W_{\boldsymbol\kappa}^{(2)}$ at
$(\mathbf a_{*,1},\mathbf a_{*,2})$ is positive definite, then, after
possibly decreasing $\rho$,
\begin{equation*}
  \sup_{t\in\bR}\ \sup_{(\mathbf b_1,\mathbf b_2)\in
  \Cset(\zeta_{1,\varepsilon,t})\times
  \Cset(\zeta_{2,\varepsilon,t})}
  |(\mathbf b_1,\mathbf b_2)
  -(\mathbf a_{*,1},\mathbf a_{*,2})|
  \leq C\varepsilon^{1/2}.
\end{equation*}
Here $C>0$ depends only on
$D,\rho$, and $\mathbf a_{*,i},\kappa_i,M_i,R_i$, $i=1,2$.
\end{corollary}

\begin{proof}
For centered radial data,
$\eta^{(2)}_{\varepsilon,0}=\varepsilon$.  The conclusions follow from
\eqref{eq:two-quantitative}, componentwise coercivity, and, when the Hessian
is positive definite, the quadratic lower bound for
$W_{\boldsymbol\kappa}^{(2)}$ together with
\eqref{eq:two-center-confinement}, exactly as in the proof of
\cref{coro_crdata}.
\end{proof}

\subsection{Proof strategy}

We first describe a four-step strategy for proving the one-vortex result.  The first three parts establish a conditional
estimate at a fixed time, and the fourth propagates it globally in time.  The
two-vortex proof follows the same scheme, with the additional interaction
term described below.

\begin{enumerate}[label=\textup{Step~\arabic*.}]
\item \emph{Scale-uniform estimates relating $\mathcal D$ and $\mathcal A$.}
We first prove that
\[
  \D(\omega_{\varepsilon,t})
  \gtrsim \A(\omega_{\varepsilon,t})^2,
\]
which shows that the logarithmic deficit controls the squared $L^1$ distance
to the family of translates of the radial rearrangement.  In the other direction,
initial localization gives
$\D(\omega_{\varepsilon,0})\lesssim\A(\omega_{\varepsilon,0})$, which connects
the qualitative initial assumptions with the quantitative theorem.

\item \emph{Logarithmic tail estimate.}
For every $\mathbf b\in\Cset(\omega_{\varepsilon,t})$ and every sufficiently
large $L$, we prove that
\[
  \int_{D\setminus B_{L\varepsilon}(\mathbf b)}
  \omega_{\varepsilon,t}
  \,\mathrm d\mathbf x
  \lesssim\frac{\D(\omega_{\varepsilon,t})}{\ln L}.
\]
Thus the support may form a long filament, but the circulation carried by the
filament remains small.

\item \emph{Regular energy estimates.}
Based on the logarithmic tail estimate, we prove that for every
$\mathbf b\in\Cset(\omega_{\varepsilon,t})$,
\[
  \Rcal(\omega_{\varepsilon,t})
  \geq\kappa^2H(\mathbf b)
  -\theta\D(\omega_{\varepsilon,t})-C\varepsilon,
\]
where $\theta>0$ can be chosen arbitrarily small and $C>0$ is independent of $\varepsilon$.
Combining this bound with equimeasurability and energy
conservation gives a \emph{conditional} estimate for the shape deficit and
the Robin-function gap, as long as the deficit is small and all optimal
centers remain in the trapping well.

\item \emph{Global propagation via a set-valued first-exit argument.}
An optimal center need not be unique, and a continuous choice of centers may
not exist.  We therefore follow the full compact set
$\Cset(\omega_{\varepsilon,t})$ and use its upper semicontinuity.
If, at a first exit time, either the deficit reached its threshold or an optimal center reached the boundary of the trapping well, the conditional estimate would keep the deficit below the threshold and all optimal centers strictly inside the well, yielding a contradiction. The set-valued first-exit argument therefore gives the
estimate for every time and every optimal center.
\end{enumerate}

For a pair of opposite-sign vortices, the additional term is the mutual interaction of
the two components.  Its sign is favorable in the Lyapunov quantity.
Positivity of the Dirichlet Green function allows us to discard tail
interactions in a lower bound, while smoothness away from the diagonal
controls the interaction of the two localized cores.
Together with the two componentwise regular-energy estimates, this yields the appropriate Kirchhoff--Routh gap for the opposite-sign pair.
The same set-valued first-exit argument then applies in $\bR^4$.

\section{Scale-uniform estimates relating \texorpdfstring{$\D$ and $\A$}{D and A}}
\label{sec:quantitative-rearrangement}

We begin with the following special case of the Yan--Yao inequality
\cite[Theorem~1.1]{YanYao2022}.

\begin{lemma}
\label{lem:yan-yao-log}
Let $f\in L^\infty(\bR^2)$ be nonzero, nonnegative, and compactly
supported. Suppose that, for some $R_*>0$,
\[
  \operatorname{supp}f^*\subset B_{R_*}(\mathbf0).
\]
Then
\begin{equation*}
  \D(f)
  \geq c_{\rm YY}R_*^{-2}\|f\|_1\|f\|_\infty^{-1}
  \A(f)^2
\end{equation*}
for some universal constant $c_{\rm YY}>0$.
\end{lemma}

\begin{proof}
This follows from \cite[Theorem~1.1]{YanYao2022} by taking $n=2$, $k=0$,
and $W(\mathbf x)=-\ln|\mathbf x|$.
\end{proof}

\begin{remark}
The result of Yan and Yao applies to a broader class of interaction
potentials and does not require $f$ to be compactly supported.
\end{remark}

The next lemma follows directly from \cref{lem:yan-yao-log}.

\begin{lemma}[Control of asymmetry by the deficit]
\label{lem:shape-coercivity}
Fix $\kappa,M,R>0$. Then there exists
$c_0=c_0(\kappa,M,R)>0$ such that, whenever
$\varepsilon>0$ and $f\in L^\infty(\bR^2)$ is nonnegative, compactly
supported, and satisfies
\[
  \|f\|_1=\kappa,\qquad
  \|f\|_\infty\leq M\varepsilon^{-2},\qquad
  \operatorname{supp}f^*\subset B_{R\varepsilon}(\mathbf0),
\]
one has
\begin{equation}\label{eq:shape-coercivity}
  \D(f)\geq c_0\A(f)^2.
\end{equation}
\end{lemma}

\begin{proof}
Just apply \cref{lem:yan-yao-log} with $R_*=R\varepsilon$ and take $c_0=c_{\rm YY}\kappa/(MR^2)$.
\end{proof}

Under an additional localization assumption, we also obtain a reverse
estimate.

\begin{lemma}[Control of the deficit by asymmetry under localized support]
\label{lem:localized-upper}
Fix $M,R>0$.  Then there exists
$C_{\rm loc}=C_{\rm loc}(M,R)>0$ such that, whenever
$\varepsilon>0$ and $f\in L^\infty(\bR^2)$ is nonnegative and satisfies
\[
  \|f\|_\infty\leq M\varepsilon^{-2},\qquad
  \operatorname{supp}f\subset B_{R\varepsilon}(\mathbf a)
\]
for some $\mathbf a\in\bR^2$,
one has
\begin{equation}\label{eq:localized-upper}
 \D(f)\leq C_{\rm loc}\A(f).
\end{equation}
\end{lemma}

\begin{proof}
The conclusion is trivial if $f\equiv0$, so assume that $f$ is nonzero.
Since $f^*$ is radially decreasing and equimeasurable with $f$, the
localized support assumption on $f$ gives
\[
  |\{f^*>0\}|=|\{f>0\}|<\pi R^2\varepsilon^2,
\]
and hence
\begin{equation}\label{suppfstar}
  \operatorname{supp}f^*\subset B_{R\varepsilon}(\mathbf0).
\end{equation}
By \cref{lem:centers}, there exists an optimal center
$\mathbf b\in\Cset(f)$ such that
$|\mathbf a-\mathbf b|\leq2R\varepsilon$.
Set $\widetilde f:=f(\cdot+\mathbf b)$ and $S:=3R$.  Then
\begin{equation}\label{supptdf}
  \operatorname{supp}\widetilde f
  \subset B_{R\varepsilon}(\mathbf a-\mathbf b)
  \subset B_{S\varepsilon}(\mathbf0),
  \end{equation}
and
\begin{equation}\label{linftytdf}
  \|\widetilde f\|_\infty=\|f^*\|_\infty
  =\|f\|_\infty\leq M\varepsilon^{-2}.
  \end{equation}
Set $q:=f^*-\widetilde f$.  Then $q$ satisfies
\begin{equation}\label{ql1a}
\int_{\bR^2}q(\mathbf x)\,\mathrm d\mathbf x=0,\qquad \operatorname{supp}q  \subset B_{S\varepsilon}(\mathbf0),\qquad
  \|q\|_1=\A(f).
\end{equation}
Moreover, we have the following identity:
\begin{equation}\label{ql1abc}
  f^*(\mathbf x)f^*(\mathbf y)
  -\widetilde f(\mathbf x)\widetilde f(\mathbf y)
  =q(\mathbf x)f^*(\mathbf y)
  +\widetilde f(\mathbf x)q(\mathbf y).
\end{equation}
Using \eqref{suppfstar}, \eqref{supptdf}, \eqref{ql1a}, and
\eqref{ql1abc}, we obtain
\begin{align*}
\D(f)&=\Pcal(f^*)-\Pcal(f)\\
&=\Pcal(f^*)-\Pcal(\widetilde f)\\
&=\frac{1}{4\pi}\iint_{B_{S\varepsilon}(\mathbf0)\times B_{S\varepsilon}(\mathbf0)}
  \ln\frac{1}{|\mathbf x-\mathbf y|}
  \bigl[q(\mathbf x)f^*(\mathbf y)
  +\widetilde f(\mathbf x)q(\mathbf y)\bigr]
  \,\mathrm d\mathbf x\,\mathrm d\mathbf y\\
&=\frac{1}{4\pi}\iint_{B_{S\varepsilon}(\mathbf0)\times B_{S\varepsilon}(\mathbf0)}
  \ln\frac{2S\varepsilon}{|\mathbf x-\mathbf y|}
  \bigl[q(\mathbf x)f^*(\mathbf y)
  +\widetilde f(\mathbf x)q(\mathbf y)\bigr]
  \,\mathrm d\mathbf x\,\mathrm d\mathbf y\\
  &\leq \frac{\A(f)}{4\pi}\left( \sup_{\mathbf x\in B_{S\varepsilon}(\mathbf0)} \int_{B_{S\varepsilon}(\mathbf0)}
  \ln\frac{2S\varepsilon}{|\mathbf x-\mathbf y|}f^*(\mathbf y)
  \,\mathrm d\mathbf y + \sup_{\mathbf y\in B_{S\varepsilon}(\mathbf0)} \int_{B_{S\varepsilon}(\mathbf0)}
  \ln\frac{2S\varepsilon}{|\mathbf x-\mathbf y|}\widetilde f(\mathbf x) \,\mathrm d\mathbf x \right).
\end{align*}
The fourth equality follows from the fact that $q$ has zero integral.
For the last inequality, we use $\|q\|_1=\A(f)$ and
\[
  \ln\frac{2S\varepsilon}{|\mathbf x-\mathbf y|}\geq0
  \qquad
  (\mathbf x,\mathbf y\in B_{S\varepsilon}(\mathbf0)).
\]
It remains to show that the two logarithmic potentials in the resulting
expression are uniformly bounded in $\varepsilon$.  For every
$\mathbf x\in B_{S\varepsilon}(\mathbf0)$, the first potential satisfies
\begin{equation*}
\begin{split}
 \int_{B_{S\varepsilon}(\mathbf0)}
  \ln\frac{2S\varepsilon}{|\mathbf x-\mathbf y|}f^*(\mathbf y)
  \,\mathrm d\mathbf y
  \leq &  M\varepsilon^{-2}\int_{B_{S\varepsilon}(\mathbf0)}
  \ln\frac{2S\varepsilon}{|\mathbf x-\mathbf y|}
  \,\mathrm d\mathbf y \\
  \leq & M\varepsilon^{-2}\int_{B_{2S\varepsilon}(\mathbf0)}
  \ln\frac{2S\varepsilon}{|\mathbf y|}
  \,\mathrm d\mathbf y \\
  =&M\varepsilon^{-2}(2S\varepsilon)^2 \int_{B_{1}(\mathbf0)}
  \ln\frac{1}{|\mathbf z|}
  \,\mathrm d\mathbf z \\
  =&2\pi MS^2.
  \end{split}
  \end{equation*}
Here the first inequality uses \eqref{linftytdf}, while the last equality
uses
\[
  \int_{B_{1}(\mathbf0)}
  \ln\frac{1}{|\mathbf z|}
  \,\mathrm d\mathbf z=\frac{\pi}{2}.
\]
Analogously, for every $\mathbf y\in B_{S\varepsilon}(\mathbf0)$,
\begin{equation*}
\int_{B_{S\varepsilon}(\mathbf0)}
  \ln\frac{2S\varepsilon}{|\mathbf x-\mathbf y|}\widetilde f(\mathbf x)
  \,\mathrm d\mathbf x \leq2\pi MS^2.
  \end{equation*}
The proof is therefore complete with $C_{\rm loc}=MS^2=9MR^2$.
\end{proof}

The preceding two lemmas also yield a global persistence result for a single nearly radial
concentrated vortex in the whole plane. For completeness, we state the result precisely
and give a detailed proof below.

\begin{theorem}[Global persistence of a single nearly radial concentrated vortex in $\bR^2$]
\label{thm:whole-plane-persistence}
Fix $\kappa,M,R>0$.  Then there exists $C=C(\kappa,M,R)>0$ such that the following
holds.  For $\varepsilon>0$, let
$\omega_{\varepsilon,0}\in L^\infty(\bR^2)$ satisfy
\begin{equation}\label{eq:whole-plane-initial-data}
  \omega_{\varepsilon,0}\geq0,\qquad
  \int_{\bR^2}\omega_{\varepsilon,0}\,\mathrm d\mathbf x=\kappa,
  \qquad
  \|\omega_{\varepsilon,0}\|_\infty\leq M\varepsilon^{-2},
  \qquad
  \operatorname{supp}\omega_{\varepsilon,0}
  \subset B_{R\varepsilon}(\mathbf0).
\end{equation}
Let $\omega_{\varepsilon,t}$, $t\in\bR$, be the corresponding Yudovich
solution of the two-dimensional Euler equation in $\bR^2$.  Then
\begin{equation}\label{eq:whole-plane-shape-persistence}
  \sup_{t\in\bR}\A(\omega_{\varepsilon,t})^2
  \leq C\A(\omega_{\varepsilon,0}).
\end{equation}
If, in addition,
\begin{equation}\label{eq:whole-plane-smallness}
  \A(\omega_{\varepsilon,0})\leq\frac{\kappa^2}{C},
\end{equation}
then, with
\[
  \mathbf X_\varepsilon
  :=\frac1\kappa\int_{\bR^2}\mathbf x
  \omega_{\varepsilon,0}(\mathbf x)\,\mathrm d\mathbf x,
\]
one has
\begin{equation}\label{eq:whole-plane-center-persistence}
  \sup_{t\in\bR}\ \sup_{\mathbf b\in\Cset(\omega_{\varepsilon,t})}
  |\mathbf b-\mathbf X_\varepsilon|
  \leq (1+\sqrt2)R\varepsilon.
\end{equation}
Consequently,
\begin{equation}\label{eq:whole-plane-center-origin}
  \sup_{t\in\bR}\ \sup_{\mathbf b\in\Cset(\omega_{\varepsilon,t})}
  |\mathbf b|
  \leq(2+\sqrt2)R\varepsilon.
\end{equation}
In particular, for any family of initial data satisfying
\eqref{eq:whole-plane-initial-data} with the same $\kappa,M,R$, if
$\A(\omega_{\varepsilon,0})\to0$ as $\varepsilon\to0$, then
\[
  \sup_{t\in\bR}\A(\omega_{\varepsilon,t})\longrightarrow0,
  \qquad
  \sup_{t\in\bR}\ \sup_{\mathbf b\in
  \Cset(\omega_{\varepsilon,t})}|\mathbf b|\longrightarrow0.
\]
\end{theorem}
\begin{proof}
Since $\omega_{\varepsilon,t}$ is transported by an area-preserving
flow, it remains compactly supported and equimeasurable with
$\omega_{\varepsilon,0}$. Hence
\[
  \omega_{\varepsilon,t}^*
  =\omega_{\varepsilon,0}^*
  \qquad\text{for every }t\in\bR.
\]
In particular,
\[
  \int_{\bR^2}\omega_{\varepsilon,t}\,\mathrm d\mathbf x=\kappa,
  \qquad
  \|\omega_{\varepsilon,t}\|_\infty
  =\|\omega_{\varepsilon,0}\|_\infty
  \leq M\varepsilon^{-2},
  \qquad
  \operatorname{supp}\omega_{\varepsilon,t}^*
  \subset B_{R\varepsilon}(\mathbf0).
\]
Moreover, the logarithmic self-interaction energy $\Pcal$ is
conserved, which implies
\[
  \D(\omega_{\varepsilon,t})
  =\D(\omega_{\varepsilon,0})
  \qquad(t\in\bR).
\]

Let $c_0=c_0(\kappa,M,R)$ and
$C_{\mathrm{loc}}=C_{\mathrm{loc}}(M,R)$ be the constants in
\cref{lem:shape-coercivity,lem:localized-upper}, respectively.
For every $t\in\bR$, applying the former lemma to
$\omega_{\varepsilon,t}$ and the latter to
$\omega_{\varepsilon,0}$ yields
\[
  c_0\A(\omega_{\varepsilon,t})^2
  \leq\D(\omega_{\varepsilon,t})
  =\D(\omega_{\varepsilon,0})
  \leq C_{\mathrm{loc}}\A(\omega_{\varepsilon,0}).
\]
This proves \eqref{eq:whole-plane-shape-persistence} with
$C=C_{\mathrm{loc}}/c_0$.

It is also well known that, for compactly supported Yudovich
solutions of the Euler equation in $\bR^2$, the mass center and the
second moment about the origin are conserved, that is, for any $t\in\mathbb R$,
\[
  \frac1\kappa\int_{\bR^2}\mathbf x
  \omega_{\varepsilon,t}(\mathbf x)\,\mathrm d\mathbf x
  =\mathbf X_\varepsilon,
  \qquad
  I_\varepsilon(t)
  :=\int_{\bR^2}|\mathbf x-\mathbf X_\varepsilon|^2
  \omega_{\varepsilon,t}(\mathbf x)\,\mathrm d\mathbf x
  =I_\varepsilon(0).
\]
According to the support condition in \eqref{eq:whole-plane-initial-data}, we have
\begin{equation}\label{aax1}
|\mathbf X_\varepsilon|\leq R\varepsilon,
\end{equation}
and
\begin{equation}\label{aax2}
  I_\varepsilon(0)
  =\int_{\bR^2}|\mathbf x|^2
  \omega_{\varepsilon,0}(\mathbf x),\mathrm d\mathbf x
  -\kappa|\mathbf X_\varepsilon|^2
  \leq\kappa R^2\varepsilon^2
\end{equation}
Besides, when \eqref{eq:whole-plane-smallness} holds,
\eqref{eq:whole-plane-shape-persistence} implies
\[
  \A(\omega_{\varepsilon,t})\leq\kappa
  \qquad\text{for every }t\in\bR.
\]
Hence, by \eqref{eq:outside-optimal-ball},
\begin{equation}\label{eq:whole-plane-local-mass}
  \int_{B_{R\varepsilon}(\mathbf b)}
  \omega_{\varepsilon,t}(\mathbf x)\,\mathrm d\mathbf x
  \geq\kappa-\frac12\A(\omega_{\varepsilon,t})
  \geq\frac\kappa2
\end{equation}
for every $t\in\bR$ and every
$\mathbf b\in\Cset(\omega_{\varepsilon,t})$.

Now fix arbitrary $t\in\bR$ and
$\mathbf b\in\Cset(\omega_{\varepsilon,t})$, and set
$d:=|\mathbf b-\mathbf X_\varepsilon|.$
It remains to prove that
\[
  d\leq(1+\sqrt2)R\varepsilon.
\]
If $d\leq R\varepsilon$, this estimate is immediate. Suppose instead
that $d>R\varepsilon$. Then for any
$\mathbf x\in B_{R\varepsilon}(\mathbf b)$, we have
$|\mathbf x-\mathbf X_\varepsilon|\geq d-R\varepsilon.$
Therefore, by \eqref{eq:whole-plane-local-mass},
\[
  I_\varepsilon(0)
  =I_\varepsilon(t)
 \geq
  \int_{B_{R\varepsilon}(\mathbf b)}
  |\mathbf x-\mathbf X_\varepsilon|^2
  \omega_{\varepsilon,t}(\mathbf x)\,\mathrm d\mathbf x \geq\frac\kappa2(d-R\varepsilon)^2,
\]
which in combination with \eqref{aax2} gives
\[
  d
  \leq R\varepsilon
  +\sqrt{\frac{2I_\varepsilon(0)}{\kappa}}
  \leq(1+\sqrt2)R\varepsilon.
\]
Since $t$ and $\mathbf b$ were arbitrary, this proves
\eqref{eq:whole-plane-center-persistence}.
Combining \eqref{eq:whole-plane-center-persistence} and \eqref{aax1},
we get \eqref{eq:whole-plane-center-origin}. The two limiting assertions follow from
\eqref{eq:whole-plane-shape-persistence} and
\eqref{eq:whole-plane-center-origin}.
\end{proof}

\section{Logarithmic tail estimates}\label{sec:estimates_tail}

In this section, we prove that, for every optimal center $\mathbf b$ and all sufficiently
large $L$, the circulation outside $B_{L\varepsilon}(\mathbf b)$ is bounded
by $C\D(f)/\ln L$.  The argument relies on a layer-cake representation of
the deficit.

For a nonnegative $f\in L^\infty(\bR^2)$ with compact support and $r>0$,
set
\[
  J_r(f):=\iint_{|\mathbf x-\mathbf y|<r}
  f(\mathbf x)f(\mathbf y)\,\mathrm d\mathbf x\,\mathrm d\mathbf y.
\]
This quantity measures the product mass of ordered point pairs separated by
less than $r$ and is natural for decomposing the logarithmic interaction
across distance scales.  Clearly, $J_r(f)\geq0$, $r\mapsto J_r(f)$ is
nondecreasing, and $\lim_{r\to\infty}J_r(f)=\|f\|_1^2$ (in fact, $J_r(f)=\|f\|_1^2$ once $r$ exceeds the diameter of $\operatorname{supp}f$).
Moreover, for every $r>0$,
\begin{equation}\label{eq:Jr-upper}
  J_r(f)
  =\int_{\bR^2}f(\mathbf x)
    \left(\int_{B_r(\mathbf x)}f(\mathbf y)\,\mathrm d\mathbf y\right)
    \,\mathrm d\mathbf x
 \leq \pi r^2\|f\|_1\|f\|_\infty.
\end{equation}

\begin{lemma}[Logarithmic layer-cake identity]
Let $f\in L^\infty(\bR^2)$ be nonnegative and compactly supported. Then
\begin{equation}\label{eq:layer-cake}
  4\pi\D(f)=\int_0^\infty
  \frac{J_r(f^*)-J_r(f)}r\,\mathrm dr.
\end{equation}
\end{lemma}

\begin{proof}
The claim is immediate if $f\equiv0$.  Otherwise, set
$\kappa:=\|f\|_1=\|f^*\|_1$ and choose $R_0$ larger than the diameters of
$\operatorname{supp}f$ and $\operatorname{supp}f^*$.  Then, for $q=f$ or
$q=f^*$, it holds that
\[
  \ln\frac{R_0}{|\mathbf x-\mathbf y|}
  =\int_0^{R_0}
  \mathbf1_{\{|\mathbf x-\mathbf y|<r\}}\,\frac{\mathrm dr}{r}
  \qquad
  \text{for a.e. }(\mathbf x,\mathbf y)
  \in\operatorname{supp}q\times\operatorname{supp}q.
\]
To justify interchanging the order of integration, we first observe that the
relevant iterated absolute integral is finite.  Indeed,
by \eqref{eq:Jr-upper} applied to $q$,
\[
  \int_0^{R_0}\iint_{\bR^2\times\bR^2}
  \mathbf1_{\{|\mathbf x-\mathbf y|<r\}}
  \frac{q(\mathbf x)q(\mathbf y)}r
  \,\mathrm d\mathbf x\,\mathrm d\mathbf y\,\mathrm dr
  =\int_0^{R_0}\frac{J_r(q)}r\,\mathrm dr
  \leq\frac12\pi R_0^2\kappa\|q\|_\infty<\infty.
\]
Fubini's theorem therefore gives
\begin{align}
  \iint_{\bR^2\times\bR^2}\ln\frac{R_0}{|\mathbf x-\mathbf y|}
  q(\mathbf x)q(\mathbf y)\,\mathrm d\mathbf x\,\mathrm d\mathbf y&=\int_0^{R_0}\iint_{\bR^2\times\bR^2}
  \mathbf1_{\{|\mathbf x-\mathbf y|<r\}}
  \frac{q(\mathbf x)q(\mathbf y)}r
  \,\mathrm d\mathbf x\,\mathrm d\mathbf y\,\mathrm dr \notag \\
  &=  \int_0^{R_0}\frac{J_r(q)}r\,\mathrm dr.\label{eq:truncated-layer-cake}
\end{align}
Since
\[
  \iint_{\bR^2\times\bR^2}\ln\frac{R_0}{|\mathbf x-\mathbf y|}
  q(\mathbf x)q(\mathbf y)\,\mathrm d\mathbf x\,\mathrm d\mathbf y
  =\kappa^2\ln R_0+4\pi\Pcal(q),
\]
subtracting the identity \eqref{eq:truncated-layer-cake} for $q=f$ from that for $q=f^*$ cancels the terms
$\kappa^2\ln R_0$,
which gives
 \[
  4\pi\D(f)=\int_0^{R_0}
  \frac{J_r(f^*)-J_r(f)}r\,\mathrm dr.
\]
Since $J_r(f)=J_r(f^*)=\kappa^2$ for $r\geq R_0$, the upper limit $R_0$  in the above integral can be replaced by $\infty$, which proves \eqref{eq:layer-cake}.
\end{proof}

The identity \eqref{eq:layer-cake} converts the separation between a
concentrated core and a distant tail into a contribution to the rearrangement
deficit.  Combining it with the control of the mass near an optimal center
yields the following logarithmic tail estimate.  This estimate will play a
key role in controlling the regular part of the energy in the next section.

\begin{proposition}[Logarithmic tail bound]\label{prop:tail}
Fix $\kappa,M,R>0$.  Then there exist $\delta_{\rm t},L_{\rm t},C_{\rm t}>0$,
depending only on $\kappa,M,R$, such that, if $\varepsilon>0$ and
$f\in L^\infty(\bR^2)$ is nonnegative, compactly supported, and satisfies
\[
  \|f\|_1=\kappa,\qquad
  \|f\|_\infty\leq M\varepsilon^{-2},\qquad
  \operatorname{supp}f^*\subset B_{R\varepsilon}(\mathbf0),\qquad
  \D(f)\leq\delta_{\rm t},
\]
then, for every
$\mathbf b\in\Cset(f)$ and every $L\geq L_{\rm t}$,
\begin{equation}\label{eq:tail}
  \int_{\bR^2\setminus B_{L\varepsilon}(\mathbf b)}f
  \,\mathrm d\mathbf x
  \leq\frac{C_{\rm t}}{\ln L}\D(f).
\end{equation}
\end{proposition}

\begin{proof}
Let $c_0$ be the constant in \eqref{eq:shape-coercivity}, and fix
$0<\delta_{\rm t}<c_0\kappa^2/4$.  Then
\eqref{eq:shape-coercivity} and $\D(f)\leq\delta_{\rm t}$ give
 $\A(f)<\kappa/2$.  For every optimal
center $\mathbf b\in\Cset(f)$, \eqref{eq:outside-optimal-ball} therefore
gives
\[
  \int_{\bR^2\setminus B_{R\varepsilon}(\mathbf b)}f
  \,\mathrm d\mathbf x
  \leq\frac12\A(f)<\frac\kappa4.
\]
It follows that
\begin{equation}\label{eq:core-mass-lower}
  \int_{B_{R\varepsilon}(\mathbf b)}f
  \,\mathrm d\mathbf x
  \geq\frac{3\kappa}{4}\qquad(\mathbf b\in\Cset(f)).
\end{equation}

Fix $\mathbf b\in\Cset(f)$ and $L>\max\{3R,1\}$, and set
\[
  K:=B_{R\varepsilon}(\mathbf b),
  \qquad
  T:=\bR^2\setminus B_{L\varepsilon}(\mathbf b).
\]
Since
$\operatorname{supp}f^*\subset
B_{R\varepsilon}(\mathbf0)$, one has
$J_r(f^*)=\kappa^2$ whenever $r\geq2R\varepsilon$.  If
$r\in[2R\varepsilon,(L-R)\varepsilon)$, point pairs in
$K\times T$ and in the transposed product $T\times K$ do not contribute to
$J_r(f)$.  Indeed, for $\mathbf x\in K$ and $\mathbf y\in T$,
\[
  |\mathbf x-\mathbf y|
  \geq|\mathbf y-\mathbf b|-|\mathbf x-\mathbf b|
  \geq(L-R)\varepsilon>r.
\]
The same is true after interchanging $\mathbf x$ and
$\mathbf y$.  Thus both $K\times T$ and $T\times K$ are contained in
$\{(\mathbf x,\mathbf y):|\mathbf x-\mathbf y|\geq r\}$.  Moreover,
$K\cap T=\varnothing$, so these two rectangles are disjoint.  Using
$J_r(f^*)=\kappa^2$, the nonnegativity of $f$, Fubini's theorem, and
\eqref{eq:core-mass-lower}, we obtain
\begin{align}
  J_r(f^*)-J_r(f)
  &=\kappa^2-J_r(f) \notag\\
  &=\iint_{|\mathbf x-\mathbf y|\geq r}
  f(\mathbf x)f(\mathbf y)
  \,\mathrm d\mathbf x\,\mathrm d\mathbf y \notag\\
  &\geq\iint_{K\times T}f(\mathbf x)f(\mathbf y)
  \,\mathrm d\mathbf x\,\mathrm d\mathbf y
  +\iint_{T\times K}f(\mathbf x)f(\mathbf y)
  \,\mathrm d\mathbf x\,\mathrm d\mathbf y \notag\\
  &=2\left(\int_K f
  \,\mathrm d\mathbf x\right)
  \left(\int_T f
  \,\mathrm d\mathbf x\right)\notag\\
  &\geq\frac{3\kappa}{2}
  \int_T f
  \,\mathrm d\mathbf x.\label{eq:tail-pair-lower}
\end{align}
The right-hand side of \eqref{eq:tail-pair-lower} is independent of $r$
throughout $[2R\varepsilon,(L-R)\varepsilon)$.  By the Riesz rearrangement
inequality applied to the kernel $\mathbf1_{B_r(\mathbf0)}$, one has
$J_r(f^*)\geq J_r(f)$ for every $r>0$.  Hence the part of the layer-cake
integral outside this interval may be discarded.  Integrating
\eqref{eq:tail-pair-lower} over the indicated interval in
\eqref{eq:layer-cake} gives
\[
  4\pi\D(f)
  \geq\frac{3\kappa}{2}
  \ln\frac{L-R}{2R}
  \int_T f
  \,\mathrm d\mathbf x.
\]
Since
\[
  \ln\frac{L-R}{2R}-\frac12\ln L\longrightarrow\infty
  \qquad\text{as }L\to\infty,
\]
we may choose $L_{\rm t}>\max\{3R,1\}$ large enough such that
\[
  \ln((L-R)/(2R))\geq\frac12\ln L
\]
for every $L\geq L_{\rm t}$.
After rearrangement, this is \eqref{eq:tail}, for example with
$C_{\rm t}=16\pi/(3\kappa)$.
\end{proof}

\begin{remark}\label{rem:tail-mass-center}
In general, the center $\mathbf b$ in \eqref{eq:tail} cannot be replaced by
the mass center
\begin{equation}\label{eq:mass-center}
  \mathbf X_f:=\frac1\kappa\int_{\bR^2}
  \mathbf x f(\mathbf x)\,\mathrm d\mathbf x.
\end{equation}
Indeed, the mass center is sensitive to a small amount of vorticity far
from the core.  If a mass $\alpha$ is moved to distance $d$ from a
concentrated core, then its contribution to the displacement of the mass center is of order $\alpha d$, whereas its contribution to the logarithmic
interaction deficit is only of order $\alpha\ln(d/\varepsilon)$.  Thus one
may have
\[
  \alpha\ln(d/\varepsilon)\longrightarrow0,
  \qquad
  \frac{\alpha d}{\varepsilon}\longrightarrow\infty,
\]
so that the deficit remains small while the mass center moves arbitrarily
far from the core.
\end{remark}

\begin{remark}
Meyer \cite[Theorem~1.2, in particular (1.17)]{Meyer2025} proves a more
general logarithmically weighted estimate for the far-field vorticity after
a core--far-field decomposition.  His estimate is centered at the mass center of the extracted core
rather than at the mass center of the entire vorticity.  By contrast, \cref{prop:tail} requires no such
decomposition: it controls the entire mass outside
$B_{L\varepsilon}(\mathbf b)$ for every $L^1$-optimal translation center
$\mathbf b\in\Cset(f)$.  This is the form needed below.
\end{remark}

\section{Regular energy estimates}\label{sec:robin}

Using the logarithmic tail estimate from the previous section, we establish
two estimates for the regular energy of a concentrated vortex.  The first is
a lower bound that replaces the distributed vortex by the Robin energy at an
optimal center, with an arbitrarily small multiple of the rearrangement
deficit and an $O(\varepsilon)$ remainder.  The second is an
$O(\varepsilon)$ approximation of the initial regular energy by the Robin
energy at the prescribed initial center.  These two estimates are key
ingredients in the Lyapunov arguments in \cref{sec:global,sec:two-vortices}.

To begin with, we collect the four scale-uniform assumptions into one
notation for later use.  For $\varepsilon,\kappa,M,R>0$, set
\begin{equation*}
  \mathscr R_{\varepsilon,\kappa,M,R}(D)
  :=\left\{f\in L^\infty(D)\;\middle|\;
  \begin{aligned}
  &f\geq0, \qquad   \|f\|_1=\kappa,\\
  &\|f\|_\infty\leq M\varepsilon^{-2},\  \operatorname{supp}f^*\subset B_{R\varepsilon}(\mathbf0)
  \end{aligned}
  \right\}.
\end{equation*}
Note that
$\mathscr R_{\varepsilon,\kappa,M,R}(D)$ is nonempty whenever
\[
  0<\varepsilon\leq\sqrt{\frac{M|D|}{\kappa}}
  \qquad\text{and}\qquad
  \kappa<\pi MR^2.
\]

\begin{proposition}[Lower bound for the regular energy]
\label{prop:robin-lower}
Fix $\kappa,M,R>0$.  Let $K$ be a nonempty compact subset of $D$ and let
$\theta>0$.  Then there exist
constants
\[
  \delta_{\rm R}=\delta_{\rm R}(\kappa,M,R)>0,
  \qquad
  \varepsilon_{\rm R}>0,
  \qquad
  C_{\rm R}>0,
\]
where $\varepsilon_{\rm R}$ and $C_{\rm R}$ depend only on
$D,K,\theta,\kappa,M,R$,
such that, whenever
$f\in\mathscr R_{\varepsilon,\kappa,M,R}(D)$,
$\mathbf b\in\Cset(f)\cap K$,
$\D(f)\leq\delta_{\rm R}$, and
$0<\varepsilon<\varepsilon_{\rm R}$,
\begin{equation}\label{eq:robin-lower}
  \Rcal(f)\geq\kappa^2H(\mathbf b)-\theta\D(f)-C_{\rm R}\varepsilon.
\end{equation}
\end{proposition}

\begin{proof}
Set $\delta_{\rm R}:=\delta_{\rm t}$, and let $L_{\rm t}$ and $C_{\rm t}$ be as in \cref{prop:tail}.  Thus $\delta_{\rm R}$, $L_{\rm t}$, and $C_{\rm t}$ depend only on
$\kappa,M,R$.  For any $L\geq L_{\rm t}$, write
\[
  N=B_{L\varepsilon}(\mathbf b),
  \qquad
  m=\int_{D\setminus N}f\,\mathrm d\mathbf x.
\]
Then, by \eqref{eq:tail},
\begin{equation}\label{eq:robin-tail-mass}
  m\leq\frac{C_{\rm t}}{\ln L}\D(f).
\end{equation}
Since $\mathbf b\in K$, it follows that
$N\Subset D$ whenever $\varepsilon$ satisfies
\begin{equation}\label{ledk1}
  L\varepsilon<d_K/2,
  \qquad d_K:=\operatorname{dist}(K,\partial D)>0.
\end{equation}
By $h(\mathbf b,\mathbf b)=2H(\mathbf b)$ and the mean value theorem, we have
\[
  |h(\mathbf x,\mathbf y)-2H(\mathbf b)|
  \leq C_K L\varepsilon
  \qquad(\mathbf x,\mathbf y\in N),
\]
for some $C_K>0$ depending only on $D$ and $K$.  Hence
\begin{equation}\label{eq:h-local}
  h(\mathbf x,\mathbf y)\geq2H(\mathbf b)-C_KL\varepsilon
  \qquad(\mathbf x,\mathbf y\in N).
\end{equation}
We also need the fact that $h$ is bounded from below in $D\times D$:
\begin{equation}\label{eq:h-global-lower}
  h(\mathbf x,\mathbf y)\geq-C_D
  \qquad ((\mathbf x,\mathbf y)\in D\times D),
\end{equation}
where $C_D>0$ depends only on $D$.

Now we are ready to estimate $\mathcal R(f)$.
Since the mass of $f$ inside $N$ is $\kappa-m$, we have
\begin{align*}
  \iint_{N\times N}f(\mathbf x)f(\mathbf y)
  \,\mathrm d\mathbf x\,\mathrm d\mathbf y
  &=(\kappa-m)^2,\\
  \iint_{(D\times D)\setminus(N\times N)}
  f(\mathbf x)f(\mathbf y)
  \,\mathrm d\mathbf x\,\mathrm d\mathbf y
  &=2 m(\kappa-m)+m^2=2\kappa m-m^2.
\end{align*}
Applying \eqref{eq:h-local} on $N\times N$ and
\eqref{eq:h-global-lower} on the complement, and using
$0\leq m\leq\kappa$, we obtain
\begin{align}
  2\Rcal(f)&=\iint_{D\times D}h(\mathbf x,\mathbf y)f(\mathbf x)f(\mathbf y)
  \,\mathrm d\mathbf x\,\mathrm d\mathbf y \notag\\
  &\geq\left(2H(\mathbf b)-C_KL\varepsilon\right)(\kappa-m)^2
  -C_D(2\kappa m-m^2) \notag\\
  &=2\kappa^2H(\mathbf b)
  +(2H(\mathbf b)+C_D)(m^2-2\kappa m)
  -C_KL\varepsilon(\kappa-m)^2 \notag\\
  &\geq2\kappa^2H(\mathbf b)
  -2\kappa m\left(2\|H\|_{L^\infty(K)}+C_D\right)
  -C_K\kappa^2L\varepsilon.\label{rflbdd}
\end{align}
Here we used
\[
  2H(\mathbf b)=h(\mathbf b,\mathbf b)\geq-C_D,
  \qquad
  (\kappa-m)^2\leq\kappa^2.
\]
From \eqref{eq:robin-tail-mass} and \eqref{rflbdd}, we get
\begin{equation}\label{rflbdd2}
  \Rcal(f)
   \geq\kappa^2H(\mathbf b)
  -\frac{\kappa C_{\rm t}
  \left(2\|H\|_{L^\infty(K)}+C_D\right)}{\ln L}\D(f)
  -\frac12C_K\kappa^2L\varepsilon.
\end{equation}
Choose $L=L(D,K,\theta,\kappa,M,R)\geq L_{\rm t}$ sufficiently large
such that
\[
  \frac{\kappa C_{\rm t}
  \left(2\|H\|_{L^\infty(K)}+C_D\right)}{\ln L}
  \leq\theta.
\]
With this choice of $L$, \eqref{rflbdd2} gives the desired coefficient of
$\D(f)$.  Finally, set
\[
  \varepsilon_{\rm R}:=\frac{d_K}{2L},
  \qquad
  C_{\rm R}:=\frac12C_K\kappa^2L.
\]
Then \eqref{ledk1} holds whenever
$0<\varepsilon<\varepsilon_{\rm R}$, and \eqref{eq:robin-lower} follows
from \eqref{rflbdd2}.
\end{proof}

\begin{lemma}[Initial approximation of the regular energy]
\label{lem:initial-robin}
Let $\kappa,R>0$, $\mathbf a_*\in D$, and $\rho>0$ satisfy
$B_\rho(\mathbf a_*)\Subset D$.
Then there exist
$\varepsilon_{\rm I}=\varepsilon_{\rm I}(D,\rho,\mathbf a_*,R)>0$ and
$C_{\rm I}=C_{\rm I}(D,\rho,\mathbf a_*,\kappa,R)>0$ such that whenever
$0<\varepsilon<\varepsilon_{\rm I}$,
$\mathbf a_\varepsilon\in B_\rho(\mathbf a_*)$, and a nonnegative
$\omega_{\varepsilon,0}\in L^\infty(D)$ satisfies
\[
  \int_D\omega_{\varepsilon,0}\,\mathrm d\mathbf x=\kappa,
  \qquad
  \operatorname{supp}\omega_{\varepsilon,0}
  \subset B_{R\varepsilon}(\mathbf a_\varepsilon),
\]
one has
\begin{equation*}
  |\Rcal(\omega_{\varepsilon,0})
  -\kappa^2H(\mathbf a_\varepsilon)|\leq C_{\rm I}\varepsilon.
\end{equation*}
\end{lemma}

\begin{proof}
Since the initial vorticity is supported in a single $O(\varepsilon)$-ball,
no tail estimate is needed here.
Since $B_\rho(\mathbf a_*)\Subset D$,
\[
  d_\rho:=\operatorname{dist}
  \left(\overline{B_\rho(\mathbf a_*)},\partial D\right)>0.
\]
Take $\varepsilon_{\rm I}=d_\rho/(2R)$, and set
\[
  K_\rho:=\left\{\mathbf x\in\overline D:
  \operatorname{dist}
  (\mathbf x,\overline{B_\rho(\mathbf a_*)})
  \leq d_\rho/2\right\}\Subset D.
\]
If $0<\varepsilon<\varepsilon_{\rm I}$ and
$\mathbf x\in B_{R\varepsilon}(\mathbf a_\varepsilon)$, then
\[
  \operatorname{dist}
  (\mathbf x,\overline{B_\rho(\mathbf a_*)})
  \leq|\mathbf x-\mathbf a_\varepsilon|
  <R\varepsilon<d_\rho/2.
\]
Thus $B_{R\varepsilon}(\mathbf a_\varepsilon)\subset K_\rho$.
Since $h$ is smooth on $D\times D$ and $K_\rho\times K_\rho$ is compact,
there exists $C_h>0$, depending only on $D$, $\rho$, and $\mathbf a_*$, such
that
\[
  \max\bigl\{|\nabla_{\mathbf x}h(\mathbf x,\mathbf y)|,
  |\nabla_{\mathbf y}h(\mathbf x,\mathbf y)|\bigr\}\leq C_h
  \qquad ((\mathbf x,\mathbf y)\in K_\rho\times K_\rho).
\]
For any $\mathbf x,\mathbf y\in B_{R\varepsilon}(\mathbf a_\varepsilon)$,
the relevant line segments remain in the convex ball
$B_{R\varepsilon}(\mathbf a_\varepsilon)\subset K_\rho$.  Applying the
mean value theorem successively in the two variables therefore gives
\begin{align*}
  |h(\mathbf x,\mathbf y)-2H(\mathbf a_\varepsilon)|
  &=|h(\mathbf x,\mathbf y)
  -h(\mathbf a_\varepsilon,\mathbf a_\varepsilon)|\\
  &\leq C_h(|\mathbf x-\mathbf a_\varepsilon|
  +|\mathbf y-\mathbf a_\varepsilon|)\\
  &\leq 2C_hR\varepsilon.
\end{align*}
By the preceding pointwise estimate and the definition of $\Rcal$ in
\eqref{eq:energy-split}, we obtain
\begin{align*}
  |\Rcal(\omega_{\varepsilon,0})
  -\kappa^2H(\mathbf a_\varepsilon)|&=\frac12\left|
  \iint_{D\times D}
  \bigl[h(\mathbf x,\mathbf y)
  -h(\mathbf a_\varepsilon,\mathbf a_\varepsilon)\bigr]
  \omega_{\varepsilon,0}(\mathbf x)
  \omega_{\varepsilon,0}(\mathbf y)
  \,\mathrm d\mathbf x\,\mathrm d\mathbf y
  \right|\\
  & \leq  C_hR\kappa^2\varepsilon.
\end{align*}
Taking $C_{\rm I}:=C_hR\kappa^2$ completes the proof.
\end{proof}

\section{Global persistence: one vortex}
\label{sec:global}

We now prove the quantitative one-vortex estimate in \cref{thm:main} by
combining the preceding fixed-time estimates with a conserved Lyapunov
quantity and a set-valued first-exit argument.  The qualitative result in
\cref{thm:intro-one} then follows from this estimate.

\subsection{Conserved Lyapunov quantity and conditional estimate}

Recall the setting of \cref{thm:main}.  The parameters $\kappa,M,R>0$ are
fixed, $\mathbf a_\varepsilon\in B_\rho(\mathbf a_*)$, and
$\omega_{\varepsilon,t}$ denotes the Yudovich solution with nonnegative
initial vorticity $\omega_{\varepsilon,0}\in L^\infty(D)$
satisfying
\[
  \int_D\omega_{\varepsilon,0}\,\mathrm d\mathbf x=\kappa,
  \qquad
  \|\omega_{\varepsilon,0}\|_\infty\leq M\varepsilon^{-2},
  \qquad
  \operatorname{supp}\omega_{\varepsilon,0}
  \subset B_{R\varepsilon}(\mathbf a_\varepsilon).
\]
Since the Euler flow is measure-preserving, for every $t\in\bR$,
\[
  \omega_{\varepsilon,t}\geq0,\qquad
  \omega_{\varepsilon,t}^*=\omega_{\varepsilon,0}^*,\qquad
  \|\omega_{\varepsilon,t}\|_1=\kappa,\qquad
  \|\omega_{\varepsilon,t}\|_\infty
  \leq M\varepsilon^{-2}.
\]
Moreover,
\[
  \operatorname{supp}\omega_{\varepsilon,t}^*
  =\operatorname{supp}\omega_{\varepsilon,0}^*
  \subset B_{R\varepsilon}(\mathbf0).
\]
In view of these transport identities,
$\omega_{\varepsilon,t}\in
\mathscr R_{\varepsilon,\kappa,M,R}(D)$ for every $t\in\bR$.

A suitable conserved quantity should capture both errors to be controlled:
the rearrangement deficit and the Robin-function gap.  This suggests defining
\begin{equation}\label{eq:Lyapunov}
  \mathcal L(t)
  :=\Pcal(\omega_{\varepsilon,0}^*)-\Ecal(\omega_{\varepsilon,t})
  -\kappa^2H(\mathbf a_*).
\end{equation}
By conservation of the kinetic energy, we have
\begin{equation}\label{eq:L-conservation}
  \mathcal L(t)=\mathcal L(0)\qquad (t\in\bR ).
\end{equation}
By \eqref{eq:energy-split}, the quantity in \eqref{eq:Lyapunov}
can be rewritten as
\begin{equation}\label{eq:L-split}
  \mathcal L(t)
  =\D(\omega_{\varepsilon,t})
  +\Rcal(\omega_{\varepsilon,t})-\kappa^2H(\mathbf a_*).
\end{equation}

We next derive a conditional estimate from \eqref{eq:L-split}.  For this
purpose, we apply \cref{prop:robin-lower} with
$K=\overline{B_\rho(\mathbf a_*)}$ and $\theta=1/2$.  Denote the
corresponding deficit threshold, concentration-scale threshold, and
remainder constant by
$\delta_{\rm R}$, $\varepsilon_{\rm R}$, and $C_{\rm R}$, respectively.
Whenever $0<\varepsilon<\varepsilon_{\rm R}$ and
\begin{equation}\label{eq:inside-hyp}
  \Cset(\omega_{\varepsilon,t})
  \subset\overline{B_\rho(\mathbf a_*)},
  \qquad
  \D(\omega_{\varepsilon,t})\leq\delta_{\rm R},
\end{equation}
according to \cref{prop:robin-lower}, we have
\[
  \Rcal(\omega_{\varepsilon,t})
  \geq\kappa^2H(\mathbf b)
  -\frac12\D(\omega_{\varepsilon,t})-C_{\rm R}\varepsilon\qquad (\mathbf b\in\Cset(\omega_{\varepsilon,t})).
\]
Substituting this
inequality into \eqref{eq:L-split}, we obtain
\begin{equation}\label{eq:L-lower}
  \mathcal L(t)+C_{\rm R}\varepsilon
  \geq\frac12\D(\omega_{\varepsilon,t})
  +\kappa^2[H(\mathbf b)-H(\mathbf a_*)].
\end{equation}
Recall that $\varepsilon_{\rm I}$ and $C_{\rm I}$ are the constants from
\cref{lem:initial-robin}.  If
$0<\varepsilon<\varepsilon_{\rm I}$, then
\begin{equation}\label{eq:L-initial}
  \mathcal L(0)
  \leq\D(\omega_{\varepsilon,0})
  +\kappa^2[H(\mathbf a_\varepsilon)-H(\mathbf a_*)]
  +C_{\rm I}\varepsilon.
\end{equation}
Taking into account \eqref{eq:L-conservation}, \eqref{eq:L-lower}, and
\eqref{eq:L-initial}, for every
$\mathbf b\in\Cset(\omega_{\varepsilon,t})$ we obtain
\[
  \frac12\D(\omega_{\varepsilon,t})
  +\kappa^2[H(\mathbf b)-H(\mathbf a_*)]
  \leq\D(\omega_{\varepsilon,0})
  +\kappa^2[H(\mathbf a_\varepsilon)-H(\mathbf a_*)]
  +(C_{\rm R}+C_{\rm I})\varepsilon.
\]
Define $C_{\rm one}:=2\max\{1,C_{\rm R}+C_{\rm I}\}$.  In view of the
preceding discussion, we obtain the following proposition.

\begin{proposition}[Conditional estimate inside the trapping well]
\label{prop:one-inside}
Let $\delta_{\rm R}$, $\varepsilon_{\rm R}$,
$\varepsilon_{\rm I}$, and $C_{\rm one}$ be the positive constants introduced above, which depend only on
$D,\rho,\mathbf a_*,\kappa,M,R$.
Suppose that
$0<\varepsilon<\min\{\varepsilon_{\rm R},\varepsilon_{\rm I}\}$ and
that, at some time $t\in\bR$, \eqref{eq:inside-hyp} holds.  Then
\begin{equation}\label{eq:inside-estimate}
   \D(\omega_{\varepsilon,t})
  +\kappa^2\sup_{\mathbf b\in\Cset(\omega_{\varepsilon,t})}
  [H(\mathbf b)-H(\mathbf a_*)]
 \leq C_{\rm one}\left\{\D(\omega_{\varepsilon,0})
  +\kappa^2[H(\mathbf a_\varepsilon)-H(\mathbf a_*)]
  +\varepsilon\right\}.
\end{equation}
\end{proposition}

\subsection{Global propagation and the quantitative estimate}

\begin{proof}[Proof of \cref{thm:main}]
We propagate the conditional estimate in \cref{prop:one-inside} by applying
the set-valued first-exit lemma in \cref{lem:first-exit}.

\medskip\noindent\textbf{Step 1: choose a deficit threshold.} By the Riesz
rearrangement inequality,
$\D(\omega_{\varepsilon,0})\geq0$, while
according to \eqref{eq:strict-well},
$H(\mathbf a_\varepsilon)-H(\mathbf a_*)\geq0$.  Thus every term in
$\eta_{\varepsilon,0}$ is nonnegative.  Let $c_0$ be the constant in
\eqref{eq:shape-coercivity} and set
\[
  \delta:=\min\left\{\frac{\delta_{\rm R}}{2},\frac{c_0\kappa^2}{2}\right\}>0.
\]
If $\D(\omega_{\varepsilon,t})\leq\delta$, then
$\D(\omega_{\varepsilon,t})\leq\delta_{\rm R}$, so the deficit condition in
\cref{prop:robin-lower} is satisfied.  Moreover,
in view of \eqref{eq:shape-coercivity}, we have
\[
  \A(\omega_{\varepsilon,t})^2
  \leq\frac{\delta}{c_0}\leq\frac{\kappa^2}{2},
  \qquad
  \A(\omega_{\varepsilon,t})<\kappa.
\]
By the center-diameter estimate \eqref{eq:center-diameter}, we therefore have
\begin{equation}\label{eq:proof-diameter}
  \operatorname{diam}\Cset(\omega_{\varepsilon,t})
  \leq2R\varepsilon
  \quad\text{if }\D(\omega_{\varepsilon,t})\leq\delta.
\end{equation}

\medskip\noindent\textbf{Step 2: set up the first-exit argument.} We use
\cref{lem:first-exit} with
$V=B_\rho(\mathbf a_*)$, $w=\kappa^2H$,
$\mathbf p_*=\mathbf a_*$,
$d(t)=\D(\omega_{\varepsilon,t})$, and
$S(t)=\Cset(\omega_{\varepsilon,t})$.  Let
$\eta_*,r_*>0$ be the thresholds determined by \cref{lem:first-exit},
which depend only on $D,\rho,\mathbf a_*,\kappa,M,R$.  Recall that \eqref{eq:strict-well} makes $\mathbf a_*$ the
unique minimum point of
$\kappa^2H$ on $\overline{B_\rho(\mathbf a_*)}$.  By
\cite[Theorem~13(ii)]{Burton2005},
\[
  t\longmapsto\omega_{\varepsilon,t}
  \quad\text{is continuous from }\bR\text{ to }L^2(D).
\]
By equimeasurability, we have
\[
  d(t)=\Pcal(\omega_{\varepsilon,0}^*)
  -\Pcal(\omega_{\varepsilon,t}),
\]
so \eqref{eq:P-Lipschitz} with $p=2$ shows that $d$ is continuous.  Since
$D$ is bounded,
\[
  \|\omega_{\varepsilon,t}-\omega_{\varepsilon,s}\|_1
  \leq |D|^{1/2}
  \|\omega_{\varepsilon,t}-\omega_{\varepsilon,s}\|_2,
\]
and hence the solution is also continuous in $L^1(D)$.  According to
\cref{lem:centers}, $S$ is upper semicontinuous and has nonempty compact values.

Recall from \eqref{eq:quantitative-smallness} that
\[
  \eta_{\varepsilon,0}
  =\D(\omega_{\varepsilon,0})
  +\kappa^2[H(\mathbf a_\varepsilon)-H(\mathbf a_*)]
  +\varepsilon.
\]
In view of the definitions of $d$, $S$, $w$, and $\mathbf p_*$ above,
\eqref{eq:inside-estimate} reads
\[
  d(t)+\sup_{\mathbf b\in S(t)}[w(\mathbf b)-w(\mathbf p_*)]
  \leq C_{\rm one}\eta_{\varepsilon,0}
\]
whenever the hypotheses of \cref{prop:one-inside} hold.

\medskip\noindent\textbf{Step 3: separate the controlled region from the
boundary.} By \eqref{eq:strict-well} and compactness of
$\partial B_\rho(\mathbf a_*)$, the boundary gap
\[
  \gamma_\rho:=\kappa^2
  \min_{\mathbf b\in\partial B_\rho(\mathbf a_*)}
  [H(\mathbf b)-H(\mathbf a_*)]>0.
\]
Choose $\eta_0>0$ sufficiently small such that
\begin{equation}\label{eq:one-threshold-choice}
  C_{\rm one}\eta_0<\min\{\delta,\eta_*\},
  \qquad
  \eta_0<\frac{\gamma_\rho}{2}.
\end{equation}
Define
\[
  K_0:=\left\{\mathbf a\in\overline{B_\rho(\mathbf a_*)}:
  \kappa^2[H(\mathbf a)-H(\mathbf a_*)]
  \leq\eta_0\right\}.
\]
Taking into account that $\mathbf a_*\in K_0$ and that, for every
$\mathbf b\in\partial B_\rho(\mathbf a_*)$,
\[
  \kappa^2[H(\mathbf b)-H(\mathbf a_*)]
  \geq\gamma_\rho>2\eta_0,
\]
we see that $K_0$ is nonempty and does not meet
$\partial B_\rho(\mathbf a_*)$.  It is closed in the compact set
$\overline{B_\rho(\mathbf a_*)}$; hence $K_0\Subset B_\rho(\mathbf a_*)$.
In particular,
\[
  d_{\rm one}:=\operatorname{dist}
  (K_0,\bR^2\setminus B_\rho(\mathbf a_*))>0.
\]
From now on, we choose $\varepsilon_0>0$ such that
\begin{equation}\label{eq:one-epsilon-choice}
  \varepsilon_0<\min\left\{
  \varepsilon_{\rm R},\varepsilon_{\rm I},
  \frac{r_*}{2R},
  \frac{d_{\rm one}}{2R}
  \right\}.
\end{equation}

\medskip\noindent\textbf{Step 4: place the initial center set inside the well.} We fix
$0<\varepsilon<\varepsilon_0$ and assume that the initial data satisfy
$\eta_{\varepsilon,0}\leq\eta_0$, as required in
\eqref{eq:quantitative-smallness}.  Since the Robin gap is bounded
above by $\eta_{\varepsilon,0}\leq\eta_0$, we have
$\mathbf a_\varepsilon\in K_0$.
Since $C_{\rm one}\geq1$ and
$C_{\rm one}\eta_0<\delta$, initial smallness also gives
\begin{equation*}
  \D(\omega_{\varepsilon,0})
  \leq\eta_{\varepsilon,0}<\delta.
\end{equation*}
By \eqref{eq:initial-center-close}, applied with $r=s=R\varepsilon$, we have
\[
  \sup_{\mathbf b\in\Cset(\omega_{\varepsilon,0})}
  |\mathbf b-\mathbf a_\varepsilon|\leq2R\varepsilon.
\]
For every such $\mathbf b$, taking into account the definition of
$d_{\rm one}$ and \eqref{eq:one-epsilon-choice}, we obtain
\[
  \operatorname{dist}(\mathbf b,
  \bR^2\setminus B_\rho(\mathbf a_*))
  \geq d_{\rm one}-|\mathbf b-\mathbf a_\varepsilon|
  \geq d_{\rm one}-2R\varepsilon>0.
\]
Consequently,
\[
  \Cset(\omega_{\varepsilon,0})
  \subset B_\rho(\mathbf a_*).
\]

\medskip\noindent\textbf{Step 5: exclude a first exit.} By
\eqref{eq:one-epsilon-choice}, the choice
$0<\varepsilon<\varepsilon_0$ ensures that
\eqref{eq:inside-estimate} holds whenever
$S(t)\subset\overline{B_\rho(\mathbf a_*)}$ and $d(t)\leq\delta$.  Moreover,
according to \eqref{eq:one-threshold-choice} and
$\eta_{\varepsilon,0}\leq\eta_0$, we have
\[
  C_{\rm one}\eta_{\varepsilon,0}
  \leq C_{\rm one}\eta_0<\eta_*,
\]
so \eqref{eq:inside-estimate} implies the conditional bound
\eqref{eq:abstract-inside}.  Whenever $d(t)\leq\delta$, in view of
\eqref{eq:proof-diameter} and \eqref{eq:one-epsilon-choice}, we have
\[
  \operatorname{diam}S(t)\leq2R\varepsilon<r_*.
\]
Finally, recall from Step~4 that $S(0)\subset B_\rho(\mathbf a_*)$ and
$d(0)<\delta$.  Taking into account the continuity and upper-semicontinuity
properties established in Step~2, all the hypotheses of
\cref{lem:first-exit} are satisfied.  Hence, for every
$t\in\bR$,
\[
  \Cset(\omega_{\varepsilon,t})\subset B_\rho(\mathbf a_*),
  \qquad
  \D(\omega_{\varepsilon,t})<\delta.
\]
Consequently, \eqref{eq:inside-estimate} holds for every $t\in\bR$, and
\eqref{eq:quantitative-main} follows with $C=C_{\rm one}$.
\end{proof}

\subsection{Proof of the qualitative one-vortex theorem}

\begin{proof}[Proof of \cref{thm:intro-one}]
Choose $\rho>0$ so that \eqref{eq:well}--\eqref{eq:strict-well} hold.
According to \eqref{eq:localized-upper} and the assumptions on the initial
data, we have
\[
  \D(\omega_{\varepsilon,0})
  \leq C_{\rm loc}\A(\omega_{\varepsilon,0})\longrightarrow0.
\]
Since $\mathbf a_\varepsilon\to\mathbf a_*$ and $H$ is continuous,
$\eta_{\varepsilon,0}\to0$.  Hence, for all sufficiently small
$\varepsilon$, the hypotheses of \cref{thm:main} are satisfied.  Taking into
account \eqref{eq:quantitative-main} and \eqref{eq:shape-coercivity}, we
obtain
\[
  \sup_{t\in\bR}\A(\omega_{\varepsilon,t})^2
  \leq C\eta_{\varepsilon,0}\longrightarrow0,
\]
which proves \eqref{eq:qualitative-conclusion1}.

Moreover, by \eqref{eq:quantitative-main}, we have
\[
  \sup_{t\in\bR}\ \sup_{\mathbf b\in
  \Cset(\omega_{\varepsilon,t})}
  [H(\mathbf b)-H(\mathbf a_*)]
  \leq C\eta_{\varepsilon,0}\longrightarrow0.
\]
In view of \eqref{eq:center-confinement}, every optimal center belongs to
$B_\rho(\mathbf a_*)$.  Since $\overline{B_\rho(\mathbf a_*)}$ is compact
and $\mathbf a_*$ is its unique minimizer of $H$ by
\eqref{eq:strict-well}, we obtain
\[
  \sup_{t\in\bR}\ \sup_{\mathbf b\in
  \Cset(\omega_{\varepsilon,t})}
  |\mathbf b-\mathbf a_*|\longrightarrow0.
\]
Indeed, otherwise we could find $r>0$, a sequence
$\varepsilon_n\downarrow0$, times $t_n\in\bR$, and centers
$\mathbf b_n\in\Cset(\omega_{\varepsilon_n,t_n})$ such that
$|\mathbf b_n-\mathbf a_*|\geq r$.  By compactness, after passing to a
subsequence, we have
$\mathbf b_n\to\mathbf b\in\overline{B_\rho(\mathbf a_*)}$.  By the preceding
estimate and the continuity of $H$, $H(\mathbf b)=H(\mathbf a_*)$, while
$|\mathbf b-\mathbf a_*|\geq r$, contradicting \eqref{eq:strict-well}.
This proves \eqref{eq:qualitative-conclusion2}.
\end{proof}

\section{Global persistence: an opposite-sign vortex pair}
\label{sec:two-vortices}

We now prove the quantitative opposite-sign two-vortex estimate in
\cref{thm:two-vortex}.  The fixed-time shape, tail, and regular-energy
estimates established above apply to each transported component separately,
while the favorable sign of the mutual interaction supplies the additional
control needed for the Lyapunov estimate and the first-exit argument.  The qualitative
two-vortex theorem then follows from this quantitative estimate.

\subsection{Componentwise estimates and the energy decomposition}

Recall the setting of \cref{thm:two-vortex}.  The circulations satisfy
$\boldsymbol\kappa=(\kappa_1,\kappa_2)$ with $\kappa_1>0>\kappa_2$, and
$M_i,R_i>0$ are fixed for $i=1,2$.
Set $\omega_{\varepsilon,0}
:=\zeta_{1,\varepsilon,0}-\zeta_{2,\varepsilon,0}$, let
$\omega_{\varepsilon,t}$ be the corresponding Yudovich solution, and let
$\zeta_{i,\varepsilon,t}$ denote the transport of the $i$th component by the
common Euler flow.  Since this flow is measure preserving, for every
$t\in\bR$ and $i=1,2$,
\[
  \zeta_{i,\varepsilon,t}\geq0,\qquad
  \zeta_{i,\varepsilon,t}^*=\zeta_{i,\varepsilon,0}^*,\qquad
  \|\zeta_{i,\varepsilon,t}\|_1=|\kappa_i|,\qquad
  \|\zeta_{i,\varepsilon,t}\|_\infty
  \leq M_i\varepsilon^{-2}.
\]
Moreover,
$\operatorname{supp}\zeta_{i,\varepsilon,t}^*
=\operatorname{supp}\zeta_{i,\varepsilon,0}^*
\subset B_{R_i\varepsilon}(\mathbf0)$.  Hence
$\zeta_{i,\varepsilon,t}$ remains in
$\mathscr R_{\varepsilon,|\kappa_i|,M_i,R_i}(D)$ globally in time.

In view of these transport identities, we may apply the static estimates
componentwise.  In particular, there
are constants $c_i>0$ such that every
$f\in\mathscr R_{\varepsilon,|\kappa_i|,M_i,R_i}(D)$ satisfies
\begin{equation}\label{eq:two-shape}
  \D(f)\geq c_i\A(f)^2.
\end{equation}
The center, diameter, and logarithmic-tail estimates likewise hold for each
component with $\kappa,M,R$ replaced by $|\kappa_i|,M_i,R_i$.

The first genuinely new point is the sign in the energy splitting.  Recall
the cross interaction from \eqref{eq:cross-interaction}.
Because the total vorticity is $f_1-f_2$,
the polarization identity for the quadratic energy gives
\begin{equation}\label{eq:two-energy-split}
  \Ecal(f_1-f_2)
  =\sum_{i=1}^2[\Pcal(f_i)-\Rcal(f_i)]
  -\mathcal I(f_1,f_2).
\end{equation}
The minus sign will become a plus sign in the Lyapunov quantity.

\subsection{Cross-interaction and position-energy estimates}

\begin{lemma}[Cross-interaction lower bound]
\label{lem:two-cross-lower}
Let $\kappa_1,\kappa_2\in\bR\setminus\{0\}$, and fix $M_i,R_i>0$ for
$i=1,2$.  Let $K$ be a nonempty compact subset of
$\Lambda_2(D)$ and let
$\theta>0$.  Then there exist $\delta_{\rm cross},\varepsilon_{\rm cross},C>0$,
depending only on
\[
  D,K,\theta,|\kappa_1|,|\kappa_2|,M_1,M_2,R_1,R_2,
\]
such that the following holds.

Suppose that
$0<\varepsilon<\varepsilon_{\rm cross}$,
\[
  f_i\in\mathscr R_{\varepsilon,|\kappa_i|,M_i,R_i}(D),
  \qquad \mathbf b_i\in\Cset(f_i), \qquad i=1,2,
\]
$(\mathbf b_1,\mathbf b_2)\in K$, and
\[
  \D(f_1)+\D(f_2)\leq\delta_{\rm cross}.
\]
\nopagebreak[4]
Then the cross interaction satisfies
\nopagebreak[4]
\begin{equation}\label{eq:two-cross-lower}
  \mathcal I(f_1,f_2)
  \geq|\kappa_1\kappa_2|G(\mathbf b_1,\mathbf b_2)
  -\theta[\D(f_1)+\D(f_2)]-C\varepsilon.
\end{equation}
\end{lemma}

\begin{proof}
For each $i=1,2$, let
$\delta_{{\rm t},i}$, $L_{{\rm t},i}$, and $C_{{\rm t},i}$ be the constants
given by \cref{prop:tail} with $\kappa=|\kappa_i|$, $M=M_i$, and
$R=R_i$.  For
$L\geq\max\{L_{{\rm t},1},L_{{\rm t},2}\}$, temporarily assume that
$\D(f_i)\leq\delta_{{\rm t},i}$ and set
\[
  N_i:=B_{L\varepsilon}(\mathbf b_i),
  \qquad
  m_i:=\int_{D\setminus N_i}f_i
  \,\mathrm d\mathbf x
  \leq\frac{C_{{\rm t},i}}{\ln L}\D(f_i).
\]
For $i=1,2$, let
\begin{equation}\label{eq:coordinate-projections}
  K_i:=\{\mathbf b_i\in D:
  (\mathbf b_1,\mathbf b_2)\in K
  \text{ for some }\mathbf b_{3-i}\in D\}
\end{equation}
denote the projection of $K$ onto the $i$th coordinate, and set
\[
  d_K:=\min\left\{
  \operatorname{dist}(K_1,\partial D),
  \operatorname{dist}(K_2,\partial D),
  \min_{(\mathbf b_1,\mathbf b_2)\in K}
  |\mathbf b_1-\mathbf b_2|\right\}>0.
\]
Fix such an $L$ and suppose that $L\varepsilon<d_K/4$.  Then, uniformly for
$(\mathbf b_1,\mathbf b_2)\in K$,
\[
  \operatorname{dist}(N_i,\partial D)
  \geq d_K-L\varepsilon>\frac{3d_K}{4},\quad i=1,2,
\]
  and
\[
  \operatorname{dist}(N_1,N_2)
  \geq d_K-2L\varepsilon>\frac{d_K}{2}.
\]
Moreover,
\[
  N_1\times N_2
  \subset
  \bigcup_{(\mathbf a_1,\mathbf a_2)\in K}
  \overline{B_{d_K/4}(\mathbf a_1)}
  \times\overline{B_{d_K/4}(\mathbf a_2)}
  \Subset\Lambda_2(D).
\]
By the mean value theorem, we can therefore find a constant $C_K>0$,
depending only on $D$ and $K$, such that
\[
  G(\mathbf x,\mathbf y)
  \geq G(\mathbf b_1,\mathbf b_2)-C_K L\varepsilon
\]
for
$(\mathbf x,\mathbf y)\in
B_{L\varepsilon}(\mathbf b_1)\times
B_{L\varepsilon}(\mathbf b_2)$.

Since $G>0$ on $\Lambda_2(D)$ by the strong maximum principle and
$f_1,f_2\geq0$,
\begin{align*}
  \mathcal I(f_1,f_2)
  &\geq\iint_{N_1\times N_2}
  G(\mathbf x,\mathbf y)f_1(\mathbf x)f_2(\mathbf y)
  \,\mathrm d\mathbf x\,\mathrm d\mathbf y\\
  &\geq\left( G(\mathbf b_1,\mathbf b_2)-C_K L\varepsilon\right)
  (|\kappa_1|-m_1)(|\kappa_2|-m_2).
\end{align*}
Expanding the right-hand side and taking into account
$0\leq m_i\leq|\kappa_i|$ and the boundedness of $G$ on $K$, we obtain
\[
  \mathcal I(f_1,f_2)
  \geq|\kappa_1\kappa_2|G(\mathbf b_1,\mathbf b_2)
  -C_1(m_1+m_2)-C_1L\varepsilon.
\]
Here $C_1$ depends only on $D$, $K$, and the circulations, and is independent
of $L$ and $\varepsilon$.

Choose $L\geq\max\{L_{{\rm t},1},L_{{\rm t},2}\}$ sufficiently large such that
\[
  \frac{C_1 C_{{\rm t},i}}{\ln L}\leq\theta,
  \qquad i=1,2.
\]
After fixing this $L$, set
\[
  \delta_{\rm cross}:=
  \min\left\{\delta_{{\rm t},1},\delta_{{\rm t},2}\right\},
  \qquad
  \varepsilon_{\rm cross}:=d_K/(4L).
\]
By the definitions of $\delta_{\rm cross}$ and
$\varepsilon_{\rm cross}$, the hypotheses of the lemma ensure that both tail
estimates are valid and that $L\varepsilon<d_K/4$.  Moreover,
\[
  C_1(m_1+m_2)
  \leq\theta[\D(f_1)+\D(f_2)].
\]
Taking $C:=C_1L$, we obtain \eqref{eq:two-cross-lower}.
\end{proof}

\begin{proposition}[Position-energy lower bound for an opposite-sign pair]
\label{prop:two-position-lower}
Let $\kappa_1>0>\kappa_2$, and fix $M_i,R_i>0$ for $i=1,2$.  Let $K$ be
a nonempty compact subset of $\Lambda_2(D)$ and let
$\theta>0$.  Then there exist $\delta_{\rm pos},\varepsilon_{\rm pos},C>0$,
depending only on
\[
  D,K,\theta,|\kappa_1|,|\kappa_2|,M_1,M_2,R_1,R_2,
\]
such that the following holds.

If $0<\varepsilon<\varepsilon_{\rm pos}$,
\[
  f_i\in\mathscr R_{\varepsilon,|\kappa_i|,M_i,R_i}(D),
  \qquad \mathbf b_i\in\Cset(f_i), \qquad i=1,2,
\]
$(\mathbf b_1,\mathbf b_2)\in K$, and
\[
  \D(f_1)+\D(f_2)\leq\delta_{\rm pos},
\]
then
\begin{align}
  \Rcal(f_1)+\Rcal(f_2)+\mathcal I(f_1,f_2)
 \geq
  W_{\boldsymbol\kappa}^{(2)}(\mathbf b_1,\mathbf b_2)
  -\theta[\D(f_1)+\D(f_2)]-C\varepsilon.
  \label{eq:two-position-lower}
\end{align}
\end{proposition}

\begin{proof}
For $i=1,2$, we take $K_i$ as in \eqref{eq:coordinate-projections}.
By \cref{prop:robin-lower}, applied with circulation $|\kappa_i|$, compact set
$K_i$, parameters $M_i$ and $R_i$, and $\theta/2$ in place of $\theta$, there are constants
$\delta_{{\rm R},i},\varepsilon_{{\rm R},i},C_{{\rm R},i}>0$ such that
\[
  \Rcal(f_i)\geq|\kappa_i|^2H(\mathbf b_i)
  -\frac{\theta}{2}\D(f_i)-C_{{\rm R},i}\varepsilon
\]
whenever
$\D(f_i)\leq\delta_{{\rm R},i}$ and
$0<\varepsilon<\varepsilon_{{\rm R},i}$.

Next, apply \cref{lem:two-cross-lower} with the compact set $K$ and the parameter
$\theta/2$, and denote the resulting constants by
$\delta_{\rm cross},\varepsilon_{\rm cross},C_{\rm cross}>0$.
Define
\[
  \begin{aligned}
  \delta_{\rm pos}
  &:=\min\left\{\delta_{\rm cross},
  \delta_{{\rm R},1},\delta_{{\rm R},2}\right\},\\
  \varepsilon_{\rm pos}
  &:=\min\left\{\varepsilon_{\rm cross},
  \varepsilon_{{\rm R},1},\varepsilon_{{\rm R},2}\right\}
  \end{aligned}
\]
and set
\[
  C:=C_{{\rm R},1}+C_{{\rm R},2}+C_{\rm cross}.
\]
In view of the definitions of $\delta_{\rm pos}$ and
$\varepsilon_{\rm pos}$, all three estimates apply.  Adding them gives
\begin{align*}
  &\Rcal(f_1)+\Rcal(f_2)+\mathcal I(f_1,f_2)\\
  \geq &\sum_{i=1}^2|\kappa_i|^2H(\mathbf b_i)
  +|\kappa_1\kappa_2|G(\mathbf b_1,\mathbf b_2)  -\theta[\D(f_1)+\D(f_2)]-C\varepsilon\\
  =&W_{\boldsymbol\kappa}^{(2)}(\mathbf b_1,\mathbf b_2)
  -\theta[\D(f_1)+\D(f_2)]-C\varepsilon.
\end{align*}
Here we have used
$|\kappa_i|^2=\kappa_i^2$ and
$|\kappa_1\kappa_2|=-\kappa_1\kappa_2$.
This proves \eqref{eq:two-position-lower}.
\end{proof}

\subsection{Lyapunov identity and control inside the trapping well}

The following conserved quantity combines the componentwise deficits, the two regular-energy terms, and the mutual interaction.  For notational convenience,
write
\[
  V_\rho:=B_\rho(\mathbf a_{*,1})\times
  B_\rho(\mathbf a_{*,2}),
  \qquad
  \mathbf a_*:=(\mathbf a_{*,1},\mathbf a_{*,2}),
  \qquad
  \mathbf a_\varepsilon:=
  (\mathbf a_{1,\varepsilon},\mathbf a_{2,\varepsilon}).
\]
For $\mathbf b=(\mathbf b_1,\mathbf b_2)\in\Lambda_2(D)$, abbreviate
\[
  \mathcal W(\mathbf b):=
  W_{\boldsymbol\kappa}^{(2)}(\mathbf b_1,\mathbf b_2),
  \qquad W_*:=\mathcal W(\mathbf a_*),
\]
and set
\[
  d_\varepsilon(t):=\sum_{i=1}^2
  \D(\zeta_{i,\varepsilon,t}),
  \qquad
  S_\varepsilon(t):=
  \Cset(\zeta_{1,\varepsilon,t})
  \times\Cset(\zeta_{2,\varepsilon,t})\subset\bR^4.
\]

\begin{lemma}[Conserved Lyapunov identity and initial control]
Under the hypotheses of \cref{thm:two-vortex}, define
\[
  \mathcal L(t)
  :=\sum_{i=1}^2\Pcal(\zeta_{i,\varepsilon,0}^*)
  -\Ecal(\omega_{\varepsilon,t})-W_*.
\]
Then
\begin{equation}\label{eq:two-L-conservation}
  \mathcal L(t)=\mathcal L(0)
  \qquad (t\in\bR),
\end{equation}
and
\begin{align}
  \mathcal L(t)
  ={}&d_\varepsilon(t)
  +\sum_{i=1}^2\Rcal(\zeta_{i,\varepsilon,t})
  +\mathcal I(\zeta_{1,\varepsilon,t},\zeta_{2,\varepsilon,t})
  -W_*.
  \label{eq:two-Lyapunov-split}
\end{align}
Moreover, there exist $\varepsilon_{\rm init},C_{\rm init}>0$, depending
only on $D,\rho$, and $\mathbf a_{*,i},\kappa_i,M_i,R_i$, $i=1,2$, such that
\begin{align}
  \mathcal L(0)
  \leq d_\varepsilon(0)
  +\mathcal W(\mathbf a_\varepsilon)
  -W_*+C_{\rm init}\varepsilon
 \leq C_{\rm init}\eta^{(2)}_{\varepsilon,0}
  \label{eq:two-L-initial}
\end{align}
whenever $0<\varepsilon<\varepsilon_{\rm init}$.
\end{lemma}

\begin{proof}
We first observe that the same measure-preserving flow transports both
components, so
$\zeta_{i,\varepsilon,t}^*=\zeta_{i,\varepsilon,0}^*$.  The kinetic energy
of the signed total vorticity is also conserved; hence $\mathcal L$ is
constant, which proves \eqref{eq:two-L-conservation}.  Replacing each
reference rearrangement by its time-$t$ counterpart
in \eqref{eq:two-energy-split} gives
\eqref{eq:two-Lyapunov-split}.

For $i=1,2$, we apply \cref{lem:initial-robin} with circulation
$|\kappa_i|$, reference point $\mathbf a_{*,i}$, and radius $\rho$; denote
the resulting constants by $\varepsilon_{{\rm I},i}$ and $C_{{\rm I},i}$.
In view of $\overline{V_\rho}\Subset\Lambda_2(D)$, the mean value theorem
also gives $\varepsilon_{\rm X},C_{\rm X}>0$ such that
\[
  |\mathcal I(\zeta_{1,\varepsilon,0},\zeta_{2,\varepsilon,0})
  -|\kappa_1\kappa_2|G(\mathbf a_{1,\varepsilon},
  \mathbf a_{2,\varepsilon})|
  \leq C_{\rm X}\varepsilon
\]
whenever $0<\varepsilon<\varepsilon_{\rm X}$.  Indeed, for sufficiently
small $\varepsilon$, all relevant support products lie in a fixed compact
subset of $\Lambda_2(D)$ on which the first derivatives of $G$ are bounded.
Set
\[
  \varepsilon_{\rm init}
  :=\min\{\varepsilon_{{\rm I},1},\varepsilon_{{\rm I},2},
  \varepsilon_{\rm X}\},
  \qquad
  C_0:=C_{{\rm I},1}+C_{{\rm I},2}+C_{\rm X}.
\]
If $0<\varepsilon<\varepsilon_{\rm init}$, the three initial estimates give
\begin{align*}
  &\sum_{i=1}^2\Rcal(\zeta_{i,\varepsilon,0})
  +\mathcal I(\zeta_{1,\varepsilon,0},\zeta_{2,\varepsilon,0})\\
   \leq&
  \sum_{i=1}^2|\kappa_i|^2H(\mathbf a_{i,\varepsilon})
  +|\kappa_1\kappa_2|G(\mathbf a_{1,\varepsilon},
  \mathbf a_{2,\varepsilon})+C_0\varepsilon\\
  =& \mathcal W(\mathbf a_\varepsilon)+C_0\varepsilon.
\end{align*}
The last equality uses
$|\kappa_i|^2=\kappa_i^2$ and
$|\kappa_1\kappa_2|=-\kappa_1\kappa_2$.
According to \eqref{eq:two-Lyapunov-split}, evaluated at $t=0$, we have
\[
  \mathcal L(0)
  \leq d_\varepsilon(0)+\mathcal W(\mathbf a_\varepsilon)-W_*
  +C_0\varepsilon.
\]
Set $C_{\rm init}:=\max\{1,C_0\}$.  The three quantities
$d_\varepsilon(0)$,
$\mathcal W(\mathbf a_\varepsilon)-W_*$, and $\varepsilon$ are nonnegative;
the second is nonnegative by \eqref{eq:two-strict-well}.  Hence
\[
  d_\varepsilon(0)+\mathcal W(\mathbf a_\varepsilon)-W_*
  +C_{\rm init}\varepsilon
  \leq C_{\rm init}\eta^{(2)}_{\varepsilon,0},
\]
which proves both inequalities in \eqref{eq:two-L-initial}.
\end{proof}

\begin{lemma}[Conditional control inside the trapping well]
\label{lem:two-inside}
There exist $\delta,\varepsilon_{\rm in},C_{\rm two}>0$, depending only on
$D,\rho$, and $\mathbf a_{*,i},\kappa_i,M_i,R_i$, $i=1,2$, such that the
following holds.

If $0<\varepsilon<\varepsilon_{\rm in}$ and, at some time $t$,
\begin{equation}\label{eq:two-inside-hyp}
  S_\varepsilon(t)\subset\overline{V_\rho},
  \qquad d_\varepsilon(t)\leq\delta,
\end{equation}
then
\begin{align}
  &d_\varepsilon(t)
  +\sup_{\mathbf b\in S_\varepsilon(t)}
  [\mathcal W(\mathbf b)-W_*]
  \leq C_{\rm two}\eta^{(2)}_{\varepsilon,0}.
  \label{eq:two-inside}
\end{align}
\end{lemma}

\begin{proof}
Apply \cref{prop:two-position-lower} with
$K=\overline{V_\rho}$ and $\theta=1/2$, and denote its constants by
$\delta_{\rm pos},\varepsilon_{\rm pos},C_{\rm pos}>0$.  With the
componentwise coercivity constants from \eqref{eq:two-shape}, set
\[
  \delta:=\frac12\min\left\{\delta_{\rm pos},
  c_1|\kappa_1|^2,c_2|\kappa_2|^2\right\}>0,
  \qquad
  \varepsilon_{\rm in}:=
  \min\left\{\varepsilon_{\rm pos},\varepsilon_{\rm init}\right\}.
\]
Suppose that \eqref{eq:two-inside-hyp} holds and let
$\mathbf b=(\mathbf b_1,\mathbf b_2)\in S_\varepsilon(t)$.  Then
$\mathbf b_i\in\Cset(\zeta_{i,\varepsilon,t})$,
$\mathbf b\in\overline{V_\rho}$, and
$d_\varepsilon(t)\leq\delta\leq\delta_{\rm pos}$.  Taking into account the
componentwise transport identities and
$\varepsilon<\varepsilon_{\rm in}\leq\varepsilon_{\rm pos}$, these facts
verify the hypotheses of \cref{prop:two-position-lower}.  Hence
\[
  \sum_{i=1}^2\Rcal(\zeta_{i,\varepsilon,t})
  +\mathcal I(\zeta_{1,\varepsilon,t},\zeta_{2,\varepsilon,t})
  \geq\mathcal W(\mathbf b)-\frac12d_\varepsilon(t)
  -C_{\rm pos}\varepsilon.
\]
According to \eqref{eq:two-Lyapunov-split}, we have
\[
  \mathcal L(t)+C_{\rm pos}\varepsilon
  \geq\frac12d_\varepsilon(t)+\mathcal W(\mathbf b)-W_*.
\]
Moreover,
\eqref{eq:two-L-conservation}, \eqref{eq:two-L-initial}, and
$\varepsilon\leq\eta^{(2)}_{\varepsilon,0}$ give
\[
  \mathcal L(t)+C_{\rm pos}\varepsilon
  =\mathcal L(0)+C_{\rm pos}\varepsilon
  \leq(C_{\rm init}+C_{\rm pos})
  \eta^{(2)}_{\varepsilon,0}.
\]
Taking the supremum over $S_\varepsilon(t)$ and using
$\mathcal W-W_*\geq0$ on $\overline{V_\rho}$, we obtain
\[
  d_\varepsilon(t)+\sup_{\mathbf b\in S_\varepsilon(t)}
  [\mathcal W(\mathbf b)-W_*]
  \leq2(C_{\rm init}+C_{\rm pos})\eta^{(2)}_{\varepsilon,0}.
\]
Thus \eqref{eq:two-inside} holds with
$C_{\rm two}:=2(C_{\rm init}+C_{\rm pos})$.
\end{proof}

\subsection{Continuation and global propagation}

\begin{proof}[Proof of \cref{thm:two-vortex}]
Let $\delta$, $\varepsilon_{\rm in}$, and $C_{\rm two}$ be as in
\cref{lem:two-inside}.  We apply the set-valued first-exit argument to the
product of the two center sets.

\medskip\noindent\textbf{Step 1: continuity and compactness properties of the relevant quantities.} By \cite[Lemma~12(iii)]{Burton2005}, each component is continuous
from $\bR$ to $L^2(D)$.  As in Step~2 of the proof of \cref{thm:main},
equimeasurability and \eqref{eq:P-Lipschitz} imply that $d_\varepsilon$ is
continuous, while boundedness of $D$ gives componentwise $L^1$ continuity.
Hence \cref{lem:centers}(vi) shows that both center maps are upper
semicontinuous with nonempty compact values.  Their finite product
$S_\varepsilon$ has the same properties.

\medskip\noindent\textbf{Step 2: control the diameter of the product center
set.} If $d_\varepsilon(t)\leq\delta$, then, in view of
\eqref{eq:two-shape} and the definition of $\delta$, for $i=1,2$ we have
\[
  \A(\zeta_{i,\varepsilon,t})^2
  \leq\frac{\D(\zeta_{i,\varepsilon,t})}{c_i}
  \leq\frac{d_\varepsilon(t)}{c_i}
  \leq\frac{|\kappa_i|^2}{2}.
\]
In particular,
$\A(\zeta_{i,\varepsilon,t})<|\kappa_i|$, so the componentwise
center-diameter estimate applies.  For any
$\mathbf b,\mathbf c\in S_\varepsilon(t)$,
\[
  |\mathbf b-\mathbf c|^2
  =|\mathbf b_1-\mathbf c_1|^2+|\mathbf b_2-\mathbf c_2|^2
  \leq4(R_1^2+R_2^2)\varepsilon^2.
\]
Set $R_{\rm pair}:=2\sqrt{R_1^2+R_2^2}$.  Then
\begin{equation}\label{eq:two-set-diameter}
  \operatorname{diam}S_\varepsilon(t)
  \leq R_{\rm pair}\varepsilon
  \quad\text{whenever }d_\varepsilon(t)\leq\delta.
\end{equation}

\medskip\noindent\textbf{Step 3: set up the first-exit argument.} We use
\cref{lem:first-exit} with
$V=V_\rho$, $w=\mathcal W$,
$\mathbf p_*=\mathbf a_*$, $d=d_\varepsilon$, and
$S=S_\varepsilon$, and denote its thresholds by $\eta_*,r_*>0$.
Recall that \eqref{eq:two-strict-well} gives the required unique minimum point
of $w$ on $\overline{V_\rho}$.  The conditional estimate required by that
lemma is precisely \eqref{eq:two-inside}.

\medskip\noindent\textbf{Step 4: choose thresholds and exclude a first exit.}
Following Steps~3--5 in the proof of \cref{thm:main}, choose $\eta_0>0$ so
small that
\begin{equation}\label{eq:two-threshold-choice}
  C_{\rm two}\eta_0<\min\{\delta,\eta_*\}
\end{equation}
and
\[
  K_0:=\left\{\mathbf b\in\overline{V_\rho}:
  \mathcal W(\mathbf b)-W_*\leq\eta_0\right\}\Subset V_\rho.
\]
Such a choice is possible by \eqref{eq:two-strict-well}.  Set
\[
  d_{\rm pair}:=\operatorname{dist}
  (K_0,\bR^4\setminus V_\rho)>0
\]
and choose $\varepsilon_0>0$ such that
\[
  \varepsilon_0<\min\left\{
  \varepsilon_{\rm in},
  \frac{r_*}{R_{\rm pair}},
  \frac{d_{\rm pair}}{R_{\rm pair}}
  \right\}.
\]

Fix $0<\varepsilon<\varepsilon_0$ and assume
\eqref{eq:two-quantitative-smallness}.  Since the three terms in
$\eta^{(2)}_{\varepsilon,0}$ are nonnegative and $C_{\rm two}\geq1$,
\[
  \mathbf a_\varepsilon\in K_0,
  \qquad
  d_\varepsilon(0)\leq\eta_0<\delta.
\]
Applying \eqref{eq:initial-center-close} to both initial components gives
\[
  \sup_{\mathbf b\in S_\varepsilon(0)}
  |\mathbf b-\mathbf a_\varepsilon|
  \leq R_{\rm pair}\varepsilon<d_{\rm pair},
\]
and hence $S_\varepsilon(0)\subset V_\rho$.  Moreover, whenever
$d_\varepsilon(t)\leq\delta$, \eqref{eq:two-set-diameter} gives
\[
  \operatorname{diam}S_\varepsilon(t)
  \leq R_{\rm pair}\varepsilon<r_*.
\]
If also $S_\varepsilon(t)\subset\overline{V_\rho}$, then
\eqref{eq:two-inside} and \eqref{eq:two-threshold-choice} yield
\[
  d_\varepsilon(t)
  +\sup_{\mathbf b\in S_\varepsilon(t)}
  [\mathcal W(\mathbf b)-W_*]
  \leq C_{\rm two}\eta^{(2)}_{\varepsilon,0}<\eta_*.
\]
Thus all the hypotheses of \cref{lem:first-exit} are satisfied, exactly as in
the one-vortex argument.  It follows that
$S_\varepsilon(t)\subset V_\rho$ and $d_\varepsilon(t)<\delta$ for every
$t\in\bR$.  This proves \eqref{eq:two-center-confinement}, and the global
validity of \eqref{eq:two-inside} gives \eqref{eq:two-quantitative} with
$C=C_{\rm two}$.
\end{proof}

\subsection{Proof of the qualitative two-vortex theorem}

\begin{proof}[Proof of \cref{thm:intro-two}]
Choose $\rho>0$ so that
$B_\rho(\mathbf a_{*,1})\times B_\rho(\mathbf a_{*,2})
\Subset\Lambda_2(D)$ and
\eqref{eq:two-strict-well} holds.  According to
\eqref{eq:localized-upper} and the assumptions on the initial data, we have
\[
  \D(\zeta_{i,\varepsilon,0})
  \leq C\A(\zeta_{i,\varepsilon,0})\longrightarrow0,
  \qquad i=1,2.
\]
Since the initial centers converge to $\mathbf a_*$ and $\mathcal W$ is
continuous, $\eta^{(2)}_{\varepsilon,0}\to0$.  Hence, for all sufficiently
small $\varepsilon$, the hypotheses of \cref{thm:two-vortex} are satisfied.
Combining \eqref{eq:two-quantitative} with \eqref{eq:two-shape} gives
\[
  \sup_{t\in\bR}\A(\zeta_{i,\varepsilon,t})^2
  \leq C\eta^{(2)}_{\varepsilon,0}\longrightarrow0,
  \qquad i=1,2,
\]
which proves \eqref{eq:two-qualitative}.

Moreover, by \eqref{eq:two-quantitative}, we have
\[
  \sup_{t\in\bR}\ \sup_{\mathbf b\in S_\varepsilon(t)}
  [\mathcal W(\mathbf b)-W_*]
  \leq C\eta^{(2)}_{\varepsilon,0}\longrightarrow0.
\]
In view of \eqref{eq:two-center-confinement}, every pair in $S_\varepsilon(t)$
belongs to $V_\rho$.  Since $\overline{V_\rho}$ is compact and $\mathbf a_*$
is its unique minimizer of $\mathcal W$ by \eqref{eq:two-strict-well}, we can repeat the
compactness argument in the proof of \cref{thm:intro-one}, with
$H$ and $\overline{B_\rho(\mathbf a_*)}$ replaced by $\mathcal W$ and
$\overline{V_\rho}$, to obtain
\[
  \sup_{t\in\bR}\ \sup_{\mathbf b\in S_\varepsilon(t)}
  |\mathbf b-\mathbf a_*|\longrightarrow0.
\]
This proves \eqref{eq:two-center-limit}.
\end{proof}

\section{Further discussion}
\label{sec:further-discussion}

\subsection{The same-sign obstruction}

The preceding argument does not extend directly from one vortex or an
opposite-sign pair to a general vortex system, because the cross interactions
enter the conserved Lyapunov quantity with different signs.

Fix a time and suppress it from the notation.  Consider $N\geq2$ signed
vortices with nonzero circulations $\kappa_i$.  For $i=1,\ldots,N$, set
$\sigma_i:=\operatorname{sgn}\kappa_i\in\{-1,1\}$, and let
$\zeta_{i,\varepsilon}\in L^\infty(D)$ be nonnegative with
\[
  \int_D\zeta_{i,\varepsilon}\,\mathrm d\mathbf x=|\kappa_i|.
\]
Set
\[
  \omega_\varepsilon:=\sum_{i=1}^N\sigma_i\zeta_{i,\varepsilon},
  \qquad
  \mathcal I_{ij}:=\mathcal I(\zeta_{i,\varepsilon},
  \zeta_{j,\varepsilon}).
\]
The quadratic energy then has the exact decomposition
\begin{equation}\label{eq:many-energy-split}
 \sum_{i=1}^N\Pcal(\zeta_{i,\varepsilon}^*)
  -\Ecal(\omega_\varepsilon)
 =\sum_{i=1}^N\D(\zeta_{i,\varepsilon})
  +\sum_{i=1}^N\Rcal(\zeta_{i,\varepsilon})
  -\sum_{1\leq i<j\leq N}\sigma_i\sigma_j\mathcal I_{ij}.
\end{equation}
Subtracting the value of $W_{\boldsymbol\kappa}^{(N)}$ at a reference
configuration only adds a state-independent constant and therefore does not
alter the signs in \eqref{eq:many-energy-split}.

For an opposite-sign pair, the cross term is $+\mathcal I_{ij}$, so the
positivity of $G$ allows the tail interactions to be discarded in a lower
bound, as in \eqref{eq:two-cross-lower}.  For a same-sign pair, however, the
term is $-\mathcal I_{ij}$, and a lower bound for the Lyapunov quantity
requires an upper bound for the interaction.
For optimal centers
$\mathbf b_i\in\Cset(\zeta_{i,\varepsilon})$ and
$\mathbf b_j\in\Cset(\zeta_{j,\varepsilon})$, one would need, for every
prescribed $\theta>0$, a constant $C_\theta>0$ and an estimate of the form
\begin{equation}\label{eq:same-sign-upper-needed}
  \mathcal I_{ij}
  \leq|\kappa_i\kappa_j|G(\mathbf b_i,\mathbf b_j)
  +\theta[\D(\zeta_{i,\varepsilon})
  +\D(\zeta_{j,\varepsilon})]+C_\theta\varepsilon
\end{equation}
uniformly while the center configuration ranges over a fixed compact
subset of $\Lambda_N(D)$.

The componentwise tail bound does not provide
\eqref{eq:same-sign-upper-needed}: it controls the remote mass but not its
distance from another component.  Indeed, a tail mass $m_\varepsilon$ lying
within $O(\varepsilon)$ of the $j$th vortex can contribute
$O(m_\varepsilon\ln(1/\varepsilon))$ to the cross interaction, which may be
of the same order as the self-interaction deficit of the $i$th component.
Taking $L\sim\varepsilon^{-1}$ in the tail estimate gives only
\[
  m_\varepsilon\ln(1/\varepsilon)
  \leq C\D(\zeta_{i,\varepsilon}),
\]
not the arbitrarily small coefficient required in
\eqref{eq:same-sign-upper-needed}.

Every collection of at least three nonzero circulations contains a same-sign
pair.  Hence the present componentwise estimates cannot control all shape
deficits together with the Kirchhoff--Routh gap.  An extension would require
a joint estimate of the form \eqref{eq:same-sign-upper-needed}, or an
additional mechanism keeping thin filaments away from the other vortices.

\subsection{Why the \texorpdfstring{$L^1$}{L1} orbital distance is natural}

The use of the $L^1$ norm in the orbital asymmetry is natural from both the
scaling of concentrated vortices and the Yan--Yao inequality.  To make the
scaling transparent, for $1\leq p<\infty$ set
\[
  \mathcal A_p(f):=\inf_{\mathbf b\in\bR^2}
  \|f(\cdot+\mathbf b)-f^*\|_p.
\]
Thus $\mathcal A_1=\A$.  If
\[
  f_\varepsilon(\mathbf x)
  =\varepsilon^{-2}g(\mathbf x/\varepsilon),
  \qquad \|g\|_1=\kappa,
\]
then
\[
  \mathcal A_p(f_\varepsilon)
  =\varepsilon^{2/p-2}\mathcal A_p(g).
\]
Thus the $L^1$ distance is invariant under the vortex-core scaling, whereas
an $L^p$ distance with $p>1$ must be multiplied by
$\varepsilon^{2-2/p}$ to become scale invariant.  The logarithmic
interaction deficit has the same invariance: since
\[
  \Pcal(f_\varepsilon)
  =\Pcal(g)+\frac{\kappa^2}{4\pi}\ln\frac1\varepsilon,
\]
the additional logarithmic term cancels between $f_\varepsilon$ and
$f_\varepsilon^*$, and hence
$\D(f_\varepsilon)=\D(g)$.

This scale matching is reflected directly in
\cref{lem:yan-yao-log}.  Under the assumptions
$R_*=R\varepsilon$, $\|f\|_1=\kappa$, and
$\|f\|_\infty\leq M\varepsilon^{-2}$, its prefactor satisfies
\[
  R_*^{-2}\|f\|_1\|f\|_\infty^{-1}
  \geq \frac{\kappa}{MR^2}.
\]
The factor $\varepsilon^{-2}$ coming from the support radius is therefore
exactly canceled by the factor $\varepsilon^2$ coming from the
$L^\infty$ bound, which gives the scale-uniform coercive estimate
$\D(f)\geq c_0\A(f)^2$.  Although normalized $L^p$ control can subsequently
be recovered by interpolation, it is weaker and uses the additional
$L^\infty$ bound.  For example, for $1<p<\infty$ and an $L^1$-optimal
translation center $\mathbf b$,
\[
  \varepsilon^{2-2/p}\mathcal A_p(f)
  \leq
  \varepsilon^{2-2/p}
  \|f(\cdot+\mathbf b)-f^*\|_p
  \leq M^{1-1/p}\A(f)^{1/p}
  \leq C\D(f)^{1/(2p)}.
\]

\subsection{Optimal translation centers versus mass centers}
\label{subsec:mass-centers}

For a nonnegative, compactly supported function
$f:\bR^2\to\bR$ with total mass $\kappa$, recall that its mass center
$\mathbf X_f$ is defined by \eqref{eq:mass-center}. When $f$ is
transported by the Euler flow, $\mathbf X_{f_t}$ varies continuously
in time and provides a natural single-valued position variable. It is
therefore natural to ask whether the set-valued optimal-center map
$\Cset(f)$ could be replaced by $\mathbf X_f$, thereby simplifying
the proof. Unfortunately, this replacement does not preserve the estimates required in our proof.  First, the tail estimate
\eqref{eq:tail} may fail if an optimal center $\mathbf b$ is replaced by the
mass center; see \cref{rem:tail-mass-center}.  The second reason concerns the regular-energy estimate \eqref{eq:robin-lower}.  By
\eqref{eq:optimal-mass-center}, shape coercivity, and the local Lipschitz
continuity of $H$, whenever $\mathbf X_f$ and $\mathbf b$ lie in the same
compact trapping region,
\begin{align*}
  \Rcal(f)
  &\geq \kappa^2H(\mathbf b)-\theta\D(f)-C_{\rm R}\varepsilon\\
  &\geq \kappa^2H(\mathbf X_f)-\theta\D(f)
  -C_{\rm R}\varepsilon-C\D(f)^{1/2}.
\end{align*}
For small $\D(f)$, the additional $\D(f)^{1/2}$ term cannot be absorbed into
 the $O(\D(f))$ term and therefore does not preserve the form of
\eqref{eq:inside-estimate}.  The same replacement in
\eqref{eq:two-position-lower} produces an analogous
$O(\D(f_1)^{1/2}+\D(f_2)^{1/2})$ error in the Kirchhoff--Routh term.

Choi, Jeong, and Yao \cite{ChoiJeongYao2024} use the mass center as their position variable rather than the set of optimal translation centers. Their setting and argument, however, differ from ours. In fact, their odd--odd
whole-plane problem reduces to a nonnegative vortex
$\rho(t)=\omega(t)\mathbf1_Q$ of unit mass in the first quadrant $Q$, and they
use
\[
  \mathbf X(t):=\int_Q\mathbf x\rho(\mathbf x,t)\,\mathrm d\mathbf x
\]
as a separate position observable.  Their shape estimate still takes the
translation-invariant form
\[
  \inf_{\mathbf a\in\bR^2}
  \|\rho(t)-\rho_0^*(\cdot-\mathbf a)\|_1\ll1,
\]
whereas their position estimate is
\[
  \operatorname{dist}(\mathbf X(t),\mathcal O_{\rm pv}(\mathbf p))\ll1.
\]
Here $\mathcal O_{\rm pv}(\mathbf p)$ denotes the corresponding odd--odd
point-vortex orbit; see \cite[Theorem~1.2]{ChoiJeongYao2024}.
In particular, the minimizing translation $\mathbf a$ is not identified with
$\mathbf X(t)$.  Their argument instead uses monotonicity properties of $\mathbf X(t)$ and
kernel estimates specific to the odd--odd configuration, neither of which is available in a general bounded domain.

\appendix

\section{Optimal translation centers}\label{app:centers}

In this appendix, we establish the properties of the full set of optimal
translation centers used in the main proofs.

\begin{lemma}[Properties of optimal translation centers]\label{lem:centers}
Let $f\in L^1(\bR^2)$ be nonzero and nonnegative, and set
$\kappa:=\|f\|_1$.  Then $\Cset(f)$ is nonempty and compact, and
$\A(f)<2\kappa$.  Moreover, the following statements hold:
\begin{enumerate}[label=\textup{(\roman*)}]
\item if $\operatorname{supp}f^*\subset B_r(\mathbf0)$ for some $r>0$,
then every $\mathbf b\in\Cset(f)$ satisfies
\begin{equation}\label{eq:outside-optimal-ball}
  \int_{\bR^2\setminus B_r(\mathbf b)}f
  \,\mathrm d\mathbf x
  \leq\frac12\A(f);
\end{equation}
\item if, for some $r,s>0$ and $\mathbf a\in\bR^2$,
\[
  \operatorname{supp}f^*\subset B_r(\mathbf0),
  \qquad
  \operatorname{supp}f\subset B_s(\mathbf a),
\]
then
\begin{equation}\label{eq:initial-center-close}
  |\mathbf b-\mathbf a|\leq r+s
  \qquad(\mathbf b\in\Cset(f));
\end{equation}
moreover, the mass center $\mathbf X_f$ defined by \eqref{eq:mass-center} satisfies
\begin{equation}\label{eq:optimal-mass-center}
  |\mathbf b-\mathbf X_f|
  \leq\frac{r+2s}{\kappa}\A(f)
  \qquad(\mathbf b\in\Cset(f));
\end{equation}
\item if $\operatorname{supp}f^*\subset B_r(\mathbf0)$ for some $r>0$
and $\A(f)<\kappa$, then
\begin{equation}\label{eq:center-diameter}
  \operatorname{diam}\Cset(f)\leq2r;
\end{equation}
\item for any nonzero, nonnegative $f,g\in L^1(\bR^2)$,
\[
  |\A(f)-\A(g)|\leq2\|f-g\|_1;
\]
\item the set-valued map $f\mapsto\Cset(f)$ has a closed graph with respect
to strong $L^1$ convergence: if $f_n,f\in L^1(\bR^2)$ are nonzero and
nonnegative, then
\[
  f_n\to f\ \text{in }L^1(\bR^2),\qquad
  \mathbf b_n\to\mathbf b\ \text{in }\bR^2,\qquad
  \mathbf b_n\in\Cset(f_n)
  \quad\Longrightarrow\quad
  \mathbf b\in\Cset(f);
\]
\item the set-valued map $f\mapsto\Cset(f)$ is upper semicontinuous\footnote{See
\cite[Section~1.4]{AubinFrankowska2009} for this standard definition.} with
respect to strong $L^1$ convergence: for every nonzero, nonnegative
$f\in L^1(\bR^2)$ and every open set $O\subset\bR^2$ with
$\Cset(f)\subset O$, there exists $\delta>0$ such that
\[
  g\in L^1(\bR^2),\quad g\geq0,\quad g\not\equiv0,\quad
  \|g-f\|_1<\delta
  \quad\Longrightarrow\quad
  \Cset(g)\subset O.
\]
\end{enumerate}
\end{lemma}

\begin{proof}
Set
\[
  \Phi_f(\mathbf b)=\|f(\cdot+\mathbf b)-f^*\|_1,
  \qquad
  I_f(\mathbf b)=\int_{\bR^2}\min\{f(\mathbf x+\mathbf b),
  f^*(\mathbf x)\}\,\mathrm d\mathbf x.
\]
Then
\begin{equation}\label{relapi0}
  \A(f)=\inf_{\mathbf b\in\bR^2}\Phi_f(\mathbf b).
\end{equation}
In view of the identity $|a-b|=a+b-2\min\{a,b\}$ for $a,b\geq0$,
\begin{equation}\label{relapi}
  \Phi_f(\mathbf b)=2\kappa-2I_f(\mathbf b).
\end{equation}
By the strong continuity of translations in $L^1$, the function
$\Phi_f:\bR^2\to\mathbb R$ is continuous.

We first prove that
\begin{equation}\label{l2ka}
  \A(f)<2\kappa.
\end{equation}
By Fubini's theorem and a change of variables,
\[
  \int_{\bR^2}\int_{\bR^2}
  f(\mathbf x+\mathbf b)f^*(\mathbf x)
  \,\mathrm d\mathbf x\,\mathrm d\mathbf b=\kappa^2>0.
\]
Thus there exists $\mathbf b\in\bR^2$ such that $I_f(\mathbf b)>0$.
Equations \eqref{relapi0} and \eqref{relapi} then give \eqref{l2ka}.

Next, we show that
\begin{equation}\label{l3ka}
  \Phi_f(\mathbf b)\to2\kappa
  \qquad\text{as }|\mathbf b|\to\infty.
\end{equation}
To this end, for any $\varepsilon>0$, choose
$R>0$ such that
\[
  \int_{\bR^2\setminus B_R(\mathbf0)}f\,\mathrm d\mathbf x
  <\frac{\varepsilon}{2},
  \qquad
  \int_{\bR^2\setminus B_R(\mathbf0)}f^*\,\mathrm d\mathbf x
  <\frac{\varepsilon}{2}.
\]
If $|\mathbf b|>2R$, then $B_R(\mathbf0)$ and
$B_R(\mathbf0)-\mathbf b$ are disjoint, and hence
\[
  \bR^2
  =\{\mathbf x:\mathbf x+\mathbf b\notin B_R(\mathbf0)\}
  \cup\bigl(\bR^2\setminus B_R(\mathbf0)\bigr).
\]
It follows that
\begin{align*}
  I_f(\mathbf b)
  &\leq\int_{\{\mathbf x:\mathbf x+\mathbf b\notin B_R(\mathbf0)\}}
  \min\{f(\mathbf x+\mathbf b),f^*(\mathbf x)\}
  \,\mathrm d\mathbf x\\
  &\quad+\int_{\bR^2\setminus B_R(\mathbf0)}
  \min\{f(\mathbf x+\mathbf b),f^*(\mathbf x)\}
  \,\mathrm d\mathbf x\\
  &\leq
  \int_{\{\mathbf x:\mathbf x+\mathbf b\notin B_R(\mathbf0)\}}
  f(\mathbf x+\mathbf b)\,\mathrm d\mathbf x
  +\int_{\bR^2\setminus B_R(\mathbf0)}f^*(\mathbf x)\,\mathrm d\mathbf x\\
  &=\int_{\bR^2\setminus B_R(\mathbf0)}f(\mathbf y)\,\mathrm d\mathbf y
  +\int_{\bR^2\setminus B_R(\mathbf0)}f^*(\mathbf x)\,\mathrm d\mathbf x\\
  &<\varepsilon.
\end{align*}
Thus $I_f(\mathbf b)\to0$ as $|\mathbf b|\to\infty$, and
\eqref{l3ka} follows from \eqref{relapi}.  By \eqref{l2ka}, \eqref{l3ka}, and the continuity of $\Phi_f$, we deduce that $\Cset(f)$, the set of global minimizers of $\Phi_f$, is nonempty and compact.

For part~(i), fix $\mathbf b\in\Cset(f)$ and write
$\widetilde f=f(\cdot+\mathbf b)$.  Set
\[
  m:=\int_{\bR^2\setminus B_r(\mathbf b)}
  f\,\mathrm d\mathbf x
  =\int_{\bR^2\setminus B_r(\mathbf0)}
  \widetilde f\,\mathrm d\mathbf x.
\]
Since $f^*$ is supported in
$B_r(\mathbf0)$,
\[
  \int_{B_r(\mathbf0)}(f^*-\widetilde f)\,\mathrm d\mathbf x=\kappa-\int_{B_r(\mathbf0)}\widetilde f\,\mathrm d\mathbf x
  =\kappa-(\kappa-m)=m,
\]
which implies
\[\int_{B_r(\mathbf0)}|f^*-\widetilde f|\,\mathrm d\mathbf x\geq m.\]
On the other hand,
\[
\int_{\bR^2\setminus B_r(\mathbf0)}|f^*-\widetilde f| \,\mathrm d\mathbf x=\int_{\bR^2\setminus B_r(\mathbf0)} \widetilde f \,\mathrm d\mathbf x=m.
\]
Therefore
$\A(f)=\|\widetilde f-f^*\|_1\geq2m$, which proves
\eqref{eq:outside-optimal-ball}.

For part~(ii), fix $\mathbf b\in\Cset(f)$.  By \eqref{relapi} and
\eqref{l2ka},
\[
  I_f(\mathbf b)
  =\kappa-\frac12\Phi_f(\mathbf b)
  =\kappa-\frac12\A(f)>0.
\]
On the other hand, the support hypotheses give
\[
  \operatorname{supp}f(\cdot+\mathbf b)
  \subset B_s(\mathbf a-\mathbf b),
  \qquad
  \operatorname{supp}f^*\subset B_r(\mathbf0).
\]
Consequently,
\[
  B_s(\mathbf a-\mathbf b)\cap B_r(\mathbf0)\neq\varnothing.
\]
For any $\mathbf x$ in this intersection,
\[
  |\mathbf a-\mathbf b|
  \leq|\mathbf a-\mathbf b-\mathbf x|+|\mathbf x|
  <s+r,
\]
which proves \eqref{eq:initial-center-close}.
For the mass-center estimate, we set
$\widetilde f=f(\cdot+\mathbf b)$.  Since $f^*$ is radial,
\[
  \int_{\bR^2}\mathbf x f^*(\mathbf x)\,\mathrm d\mathbf x=\mathbf0,
  \qquad
  \int_{\bR^2}\mathbf x\widetilde f(\mathbf x)\,\mathrm d\mathbf x
  =\kappa(\mathbf X_f-\mathbf b).
\]
Moreover, according to \eqref{eq:initial-center-close}, we have
\[
  \operatorname{supp}\widetilde f
  \subset B_s(\mathbf a-\mathbf b)
  \subset B_{r+2s}(\mathbf0),
\]
and $f^*$ is supported in the same ball.  Therefore
\[
  \kappa|\mathbf X_f-\mathbf b|
  \leq\int_{\bR^2}|\mathbf x|
  |\widetilde f(\mathbf x)-f^*(\mathbf x)|\,\mathrm d\mathbf x
  \leq(r+2s)\A(f),
\]
which proves \eqref{eq:optimal-mass-center}.

For part~(iii), take $\mathbf b_1,\mathbf b_2\in\Cset(f)$.  By
\eqref{eq:outside-optimal-ball} and the assumption $\A(f)<\kappa$,
\[
  \int_{B_r(\mathbf b_i)}f\,\mathrm d\mathbf x
  =\kappa-\int_{\bR^2\setminus B_r(\mathbf b_i)}f\,\mathrm d\mathbf x
  \geq\kappa-\frac12\A(f)>\frac\kappa2,
  \qquad i=1,2.
\]
We claim that $B_r(\mathbf b_1)\cap B_r(\mathbf b_2)\neq\varnothing$. In fact, if the two balls were disjoint, then
\[
  \kappa
  \geq\int_{B_r(\mathbf b_1)\cup B_r(\mathbf b_2)}f\,\mathrm d\mathbf x
  =\sum_{i=1}^2\int_{B_r(\mathbf b_i)}f\,\mathrm d\mathbf x
  >\kappa,
\]
a contradiction.
Taking some $\mathbf x\in B_r(\mathbf b_1)\cap B_r(\mathbf b_2)$, we have
\[
  |\mathbf b_1-\mathbf b_2|
  \leq|\mathbf b_1-\mathbf x|+|\mathbf x-\mathbf b_2|<2r.
\]
Since $\mathbf b_1$ and $\mathbf b_2$ were arbitrary, this proves
\eqref{eq:center-diameter}.

For part~(iv), take nonzero, nonnegative $f,g\in L^1(\bR^2)$ and fix
$\mathbf c\in\bR^2$.  The reverse triangle inequality and the $L^1$
contraction property of symmetric decreasing rearrangement
\cite[Theorem~3.5]{LiebLoss2001} give
\begin{align*}
  |\Phi_f(\mathbf c)-\Phi_g(\mathbf c)|
  &\leq\|f(\cdot+\mathbf c)-g(\cdot+\mathbf c)\|_1
  +\|f^*-g^*\|_1\\
  &\leq2\|f-g\|_1.
\end{align*}
Consequently,
\[
  \A(f)=\inf_{\mathbf c\in\bR^2}\Phi_f(\mathbf c)
  \leq\inf_{\mathbf c\in\bR^2}
  \bigl(\Phi_g(\mathbf c)+2\|f-g\|_1\bigr)
  =\A(g)+2\|f-g\|_1.
\]
Interchanging $f$ and $g$ yields
\[
  |\A(f)-\A(g)|\leq2\|f-g\|_1.
\]

For part~(v), suppose that $f_n$ and $f$ are nonzero and nonnegative and
\[
  f_n\to f\ \text{in }L^1(\bR^2),
  \qquad
  \mathbf b_n\in\Cset(f_n),
  \qquad
  \mathbf b_n\to\mathbf b.
\]
By the strong continuity of translations in $L^1$,
\[
  \|f_n(\cdot+\mathbf b_n)-f(\cdot+\mathbf b)\|_1
  \leq\|f_n-f\|_1
  +\|f(\cdot+\mathbf b_n)-f(\cdot+\mathbf b)\|_1
  \longrightarrow0.
\]
The rearrangement contraction also gives
\[
  \|f_n^*-f^*\|_1\leq\|f_n-f\|_1\longrightarrow0.
\]
Hence
\[
  |\Phi_{f_n}(\mathbf b_n)-\Phi_f(\mathbf b)|
  \leq\|f_n(\cdot+\mathbf b_n)-f(\cdot+\mathbf b)\|_1
  +\|f_n^*-f^*\|_1
  \longrightarrow0.
\]
Part~(iv) further implies
\[
  |\A(f_n)-\A(f)|\leq2\|f_n-f\|_1\longrightarrow0.
\]
Since $\Phi_{f_n}(\mathbf b_n)=\A(f_n)$, it follows that
\[
  \Phi_f(\mathbf b)
  =\lim_{n\to\infty}\Phi_{f_n}(\mathbf b_n)
  =\lim_{n\to\infty}\A(f_n)=\A(f).
\]
Thus $\mathbf b\in\Cset(f)$, which proves the closed-graph property.

Finally, we prove part~(vi).  Let $O$ be an open neighborhood of
$\Cset(f)$.  If the claim failed at $f$, there would exist nonzero,
nonnegative $f_n\in L^1(\bR^2)$ and $\mathbf b_n\in\bR^2$ such that
\[
  f_n\to f\ \text{in }L^1(\bR^2),
  \qquad
  \mathbf b_n\in\Cset(f_n)\setminus O.
\]
The sequence $(\mathbf b_n)$ must be bounded.  Otherwise, after passing to a
subsequence, $|\mathbf b_n|\to\infty$.  Since
\[
  |\Phi_{f_n}(\mathbf b_n)-\Phi_f(\mathbf b_n)|
  \leq2\|f_n-f\|_1
\]
and $\Phi_f(\mathbf b_n)\to2\kappa$, optimality gives
\begin{align*}
  |\A(f_n)-2\kappa|
  &=|\Phi_{f_n}(\mathbf b_n)-2\kappa|\\
  &\leq2\|f_n-f\|_1
  +|\Phi_f(\mathbf b_n)-2\kappa|
  \longrightarrow0.
\end{align*}
On the other hand, part~(iv) gives
\[
  |\A(f_n)-\A(f)|\leq2\|f_n-f\|_1\longrightarrow0,
\]
and hence $\A(f)=2\kappa$, contradicting \eqref{l2ka}.  Therefore
$(\mathbf b_n)$ is bounded.  After passing to a subsequence, let
$\mathbf b_n\to\mathbf b$.  Since $O^c$ is closed and part~(v) gives the
closed-graph property,
\[
  \mathbf b\in O^c,
  \qquad
  \mathbf b\in\Cset(f)\subset O,
\]
a contradiction.  The map is therefore upper semicontinuous.
\end{proof}

\section{Set-valued first-exit lemma}

The following set-valued continuation lemma is used in the one- and
two-vortex proofs and avoids choosing an optimal center continuously.

\begin{lemma}[Set-valued first-exit lemma]\label{lem:first-exit}
Let $V\subset\bR^m$ be bounded and open, and suppose that
$w\in C(\overline V)$ has a unique minimum point
$\mathbf p_*\in V$ on $\overline V$.  Let
$d:\bR\to[0,\infty)$ be continuous, and let $S$ be a set-valued
map from $\bR$ to $\bR^m$ with nonempty compact values.  Fix
$\delta>0$.  Then there exist $\eta_*,r_*>0$, depending only on $V,w,\delta$,
with the following property.

Assume that $S$ is upper semicontinuous\footnote{As in
\cref{lem:centers}(vi), this means that, for every $t_0\in\bR$ and every open
set $O\subset\bR^m$ with $S(t_0)\subset O$, there exists $\sigma>0$ such that
$S(t)\subset O$ whenever $|t-t_0|<\sigma$.} and that
\[
  \operatorname{diam}S(t)\leq r_*
  \quad\text{whenever }d(t)\leq\delta,
\]
and that
\begin{equation}\label{eq:abstract-inside}
  d(t)+\sup_{\mathbf b\in S(t)}[w(\mathbf b)-w(\mathbf p_*)]
  \leq\eta_*
\end{equation}
whenever $S(t)\subset\overline V$ and $d(t)\leq\delta$.
Then $S(0)\subset V$ and $d(0)<\delta$ imply that
$S(t)\subset V$ and $d(t)<\delta$ for every $t\in\bR$.
\end{lemma}

\begin{proof}
Set
\[
  \gamma:=\min_{\partial V}[w-w(\mathbf p_*)]>0
\]
and choose
\[0<\eta_*<\min\left\{\frac\delta2,\frac\gamma4\right\}.\]  Then
\[
 K_*:=\left\{\mathbf x\in\overline V:
  w(\mathbf x)-w(\mathbf p_*)\leq2\eta_*\right\}\Subset V.
\]
Set
\[
  d_*:=\operatorname{dist}(K_*,\bR^m\setminus V)>0
\]
and choose
\[
  0<r_*<d_*/2.
\]

We shall use the following local persistence observation.  If, at some
$t_0\in\bR$,
\[
  S(t_0)\subset V,
  \qquad d(t_0)<\delta,
\]
then compactness gives
\[
  \rho_{t_0}:=\operatorname{dist}
  (S(t_0),\bR^m\setminus V)>0,
  \qquad
  \epsilon_{t_0}:=\delta-d(t_0)>0.
\]
By upper semicontinuity of $S$ and continuity of $d$, there exists
$\sigma_{t_0}>0$ such that, whenever $|t-t_0|<\sigma_{t_0}$,
\[
  S(t)\subset
  \{\mathbf x:\operatorname{dist}(\mathbf x,S(t_0))<\rho_{t_0}/2\}
  \subset V,
  \qquad
  |d(t)-d(t_0)|<\epsilon_{t_0}/2.
\]
In particular, $S(t)\subset V$ and $d(t)<\delta$ throughout this
neighborhood of $t_0$.

We first consider nonnegative times and define
\[
  T_*:=\sup\left\{T>0:
  S(t)\subset V\ \text{and}\ d(t)<\delta
  \ \text{for every }t\in[0,T)\right\}.
\]
The local persistence observation at $t_0=0$ shows that $T_*>0$.

 Suppose, for contradiction, that $T_*<\infty$.
For every $t<T_*$, according to \eqref{eq:abstract-inside}, we have
\[
  d(t)\leq\eta_*<\delta/2,
  \qquad
  S(t)\subset\{w-w(\mathbf p_*)\leq\eta_*\}\subset K_*.
\]
Choose $t_n\uparrow T_*$ and $\mathbf b_n\in S(t_n)$.  By the compactness
of $K_*$, after passing to a subsequence, we have
$\mathbf b_n\to\overline{\mathbf b}\in K_*$.  Recall that an upper semicontinuous, compact-valued map has a closed graph
\cite[Proposition~1.4.8]{AubinFrankowska2009}; hence
$\overline{\mathbf b}\in S(T_*)$.  Continuity of $d$ gives
\[
  d(T_*)=\lim_{n\to\infty}d(t_n)\leq\eta_*<\delta.
\]
In view of this inequality, the diameter hypothesis applies at $T_*$, and
every $\mathbf c\in S(T_*)$ satisfies
\[
  |\mathbf c-\overline{\mathbf b}|
  \leq\operatorname{diam}S(T_*)\leq r_*<d_*/2.
\]
Since $\overline{\mathbf b}\in K_*$, it follows more precisely that
\[
  \operatorname{dist}(\mathbf c,\bR^m\setminus V)
  \geq\operatorname{dist}(\overline{\mathbf b},\bR^m\setminus V)
  -|\mathbf c-\overline{\mathbf b}|
  >d_*/2.
\]
Thus $S(T_*)\subset V$.  Since $d(T_*)<\delta$, the local persistence
observation at $t_0=T_*$ extends both controlled inequalities beyond $T_*$,
contradicting its definition.  Thus $T_*=\infty$.

Applying the same
argument to $d(-t)$ and $S(-t)$, we obtain the conclusion for negative times.
\end{proof}

\section{An explicit domain with a trapped opposite-sign pair}

In this appendix, we construct an explicit noncircular perturbation of the
disk such that the Kirchhoff--Routh function associated with a pair of equal
and opposite vortices admits a strict nondegenerate local minimum.  Throughout,
we identify $\bR^2$ with $\mathbb C$.

\begin{proposition}[A trapped opposite-sign pair in an explicit domain]
\label{prop:two-example}
Let $B:=\{z\in\mathbb C:|z|<1\}$ be the unit disk centered at the origin
and, for $\tau>0$, set
\[
  F_\tau(z):=z+\tau z^3,
  \qquad D_\tau:=F_\tau(B).
\]
Then
\begin{itemize}
\item[(i)] For every $\tau\in(0,1/3)$, the domain $D_\tau$ is a bounded,
simply connected, noncircular domain with analytic boundary.
\item[(ii)] There exist $\tau_0\in(0,1/3)$ and a family
$\{r_\tau\}_{0<\tau<\tau_0}\subset(0,1)$ such that
\[
  r_\tau\longrightarrow r_0:=\sqrt{\sqrt5-2}
  \quad\text{as }\tau\to0,
\]
and, for every $\tau\in(0,\tau_0)$ and every $\mu>0$, the function
$W_{\boldsymbol\kappa}^{(2)}$ corresponding to
$\kappa_1=\mu=-\kappa_2$ has a strict nondegenerate local minimum point at
\[
  \bigl(F_\tau(r_\tau),F_\tau(-r_\tau)\bigr).
\]
\end{itemize}
\end{proposition}

\begin{proof}
We divide the proof into four steps.

\medskip\noindent\textbf{Step 1: Geometry of \(D_\tau\).}
Suppose that \(0<\tau<1/3\). For \(z,w\in\overline B\),
\[
F_\tau(z)-F_\tau(w)
=(z-w)\bigl(1+\tau(z^2+zw+w^2)\bigr).
\]
Since
\[
\left|1+\tau(z^2+zw+w^2)\right|
\geq
1-\tau\bigl(|z|^2+|z||w|+|w|^2\bigr)
\geq 1-3\tau>0,
\]
we see that \(F_\tau\) is injective on \(\overline B\). Moreover, for
\(z\in\overline B\),
\[
|F_\tau'(z)|
=
|1+3\tau z^2|
\geq 1-3\tau|z|^2
\geq 1-3\tau>0.
\]
Thus \(F_\tau\) is a conformal bijection from \(B\) onto \(D_\tau\).
Since \(F_\tau\) is continuous and injective on the compact set
\(\overline B\), it is a homeomorphism from \(\overline B\) onto
\(F_\tau(\overline B)\). Since \(B\) is dense in \(\overline B\),
\(D_\tau=F_\tau(B)\) is dense in \(F_\tau(\overline B)\). The latter
set is compact, hence closed, and therefore
\[
\overline{D_\tau}=F_\tau(\overline B).
\]
Moreover, since \(D_\tau\) is open and \(F_\tau\) is injective on
\(\overline B\),
\[
\partial D_\tau
=\overline{D_\tau}\setminus D_\tau
=F_\tau(\overline B)\setminus F_\tau(B)
=F_\tau(\partial B).
\]
The restriction \(F_\tau|_{\partial B}\) is an injective real-analytic
immersion, and hence \(\partial D_\tau\) is an analytic Jordan curve.
Consequently, \(D_\tau\) is bounded and simply connected with analytic
boundary.

It remains to show that \(D_\tau\) is not a disk. Since \(F_\tau\) is
odd, we have
\[
D_\tau=-D_\tau.
\]
If \(D_\tau=B_R(c)\) for some \(c\in\mathbb C\), then
\[
B_R(c)=D_\tau=-D_\tau=B_R(-c),
\]
which implies \(c=0\). On the other hand,
\[
|F_\tau(e^{i\theta})|^2
=
1+\tau^2+2\tau\cos(2\theta),
\]
which is not constant in \(\theta\). Thus \(D_\tau\) is noncircular.

\medskip\noindent\textbf{Step 2: Conformal pullback.}
Let \(G_{D_\tau}\) and \(H_{D_\tau}\) denote the Dirichlet Green
function and the Robin function of \(D_\tau\), respectively.
Recall that the Dirichlet Green function and the Robin function of the
unit disk \(B\) are given by
\[
G_B(z,w)
=
\frac{1}{2\pi}
\ln\left|
\frac{1-\overline{w}z}{z-w}
\right|,
\qquad z,w\in B,\quad z\neq w,
\]
and
\[
H_B(z)
=
-\frac{1}{4\pi}\ln(1-|z|^2).
\]
By conformal invariance of the Dirichlet Green function and the
corresponding transformation rule for the Robin function
(see, for example, \cite[Section~4.4]{Ransford1995}), we have
\[
G_{D_\tau}\bigl(F_\tau(z),F_\tau(w)\bigr)
=
G_B(z,w)
=
\frac{1}{2\pi}
\ln\left|
\frac{1-\overline{w}z}{z-w}
\right|
\]
and
\[
H_{D_\tau}\bigl(F_\tau(z)\bigr)
=
H_B(z)-\frac{1}{4\pi}\ln|F_\tau'(z)|
=
-\frac{1}{4\pi}
\ln\Bigl[(1-|z|^2)|1+3\tau z^2|\Bigr].
\]
For \(\kappa_1=\mu=-\kappa_2\), the Kirchhoff--Routh function takes the
form
\[
W_{\boldsymbol\kappa}^{(2)}(a_1,a_2)
=
\mu^2 H_{D_\tau}(a_1)
+\mu^2 H_{D_\tau}(a_2)
+\mu^2 G_{D_\tau}(a_1,a_2).
\]
Define its normalized pullback by
\[
\Phi_\tau(z_1,z_2)
:=
\frac{4\pi}{\mu^2}
W_{\boldsymbol\kappa}^{(2)}
\bigl(F_\tau(z_1),F_\tau(z_2)\bigr).
\]
Using the preceding formulas and
$F_\tau'(z)=1+3\tau z^2$,
we obtain
\begin{equation}\label{eq:Phi-tau}
\Phi_\tau(z_1,z_2)
=
-\sum_{i=1}^2\ln(1-|z_i|^2)
+2\ln\left|
\frac{1-\overline{z_2}z_1}{z_1-z_2}
\right|
-\sum_{i=1}^2\ln|1+3\tau z_i^2|.
\end{equation}
The last term in \eqref{eq:Phi-tau} is precisely the contribution
of the conformal perturbation.

\medskip\noindent\textbf{Step 3: Critical point in the disk.}
We first consider the disk case \(\tau=0\). Restricting
\(\Phi_0\) to the symmetric real configurations
\[
(z_1,z_2)=(r,-r),
\]
we obtain
\[
\phi_0(r)
:=
\Phi_0(r,-r)
=
-2\ln(1-r^2)
+
2\ln\frac{1+r^2}{2r}.
\]
Its derivative is
\[
\phi_0'(r)
=
\frac{4r}{1-r^2}
+
\frac{4r}{1+r^2}
-
\frac{2}{r}.
\]
Hence
\[
\phi_0'(r)=0
\quad\Longleftrightarrow\quad
r^4+4r^2-1=0.
\]
Writing \(u=r^2\), the last equation becomes
\[
u^2+4u-1=0.
\]
The unique solution of this equation in \((0,1)\) is
\[
u_0=\sqrt5-2.
\]
Consequently, the unique \(r\in(0,1)\) satisfying
\(\phi_0'(r)=0\) is
\[
r_0=\sqrt{u_0}=\sqrt{\sqrt5-2}.
\]
Moreover,
\[
\phi_0''(r_0)
=
\frac{4(1+r_0^2)}{(1-r_0^2)^2}
+
\frac{4(1-r_0^2)}{(1+r_0^2)^2}
+
\frac{2}{r_0^2}
>0.
\]
We next verify that \((r_0,-r_0)\) is a critical point of the full
four-dimensional function \(\Phi_0\). The symmetries
\[
(z_1,z_2)\mapsto
(\overline{z_1},\overline{z_2}),
\qquad
(z_1,z_2)\mapsto(-z_2,-z_1)
\]
imply that
\[
\nabla\Phi_0(r_0,-r_0)
\in\operatorname{span}\{e_r\},
\qquad
e_r:=(1,0,-1,0).
\]
Since
\[
\nabla\Phi_0(r_0,-r_0)\cdot e_r
=
\phi_0'(r_0)=0,
\]
we obtain
\[
\nabla\Phi_0(r_0,-r_0)=0.
\]
To examine the Hessian, in the real coordinates
\((x_1,y_1,x_2,y_2)\), introduce the four mutually orthogonal
directions
\[
e_r=(1,0,-1,0),\qquad
e_c=(1,0,1,0),
\]
\[
e_v=(0,1,0,1),\qquad
e_\theta=(0,1,0,-1).
\]
The same two symmetries show that these four one-dimensional subspaces
are invariant under the Hessian at \((r_0,-r_0)\), so all mixed
entries between them vanish. The corresponding affine paths are
\[
(r+s,-r-s),\qquad
(r+s,-r+s),\qquad
(r+is,-r+is),\qquad
(r+is,-r-is).
\]
Substituting these paths into \eqref{eq:Phi-tau}, differentiating twice
at \(s=0\), and using
\[
r_0^4+4r_0^2-1=0,
\]
we obtain
\[
D^2\Phi_0(r_0,-r_0)[e_r,e_r]
=
\phi_0''(r_0)>0,
\]
\[
D^2\Phi_0(r_0,-r_0)[e_c,e_c]
=
4\left[
\frac{1+r_0^2}{(1-r_0^2)^2}
-\frac{1}{1+r_0^2}
\right]>0,
\]
\[
D^2\Phi_0(r_0,-r_0)[e_v,e_v]
=
4\left[
\frac{1}{1-r_0^2}
-\frac{1-r_0^2}{(1+r_0^2)^2}
\right]>0,
\]
whereas
\[
D^2\Phi_0(r_0,-r_0)[e_\theta,e_\theta]=0.
\]
Thus the rotational direction is the only degenerate direction.
Since $\Phi_0$ is invariant under simultaneous rotations, the
rotational orbit through $(r_0,-r_0)$ consists of critical points.
At $(r_0,-r_0)$, the tangent space to this orbit is
$\operatorname{span}\{e_\theta\}$, which coincides with the kernel of
$D^2\Phi_0(r_0,-r_0)$, while the Hessian is positive definite on its
orthogonal complement. By rotational invariance, the same properties
hold along the entire orbit.

\medskip\noindent\textbf{Step 4: Breaking rotational degeneracy.}
We now study the critical point for small \(\tau>0\). On the
symmetric real configurations,
\[
\phi_\tau(r)
:=
\Phi_\tau(r,-r)
=
\phi_0(r)-2\ln(1+3\tau r^2),
\]
and hence
\[
\partial_r\phi_\tau(r)
=
\phi_0'(r)
-
\frac{12\tau r}{1+3\tau r^2}.
\]
Set
\[
g(\tau,r):=\partial_r\phi_\tau(r).
\]
Then
\[
g(0,r_0)=0,
\qquad
\partial_rg(0,r_0)=\phi_0''(r_0)>0.
\]
Hence the implicit function theorem yields \(\tau_1>0\) and a smooth branch
\(r_\tau\in(0,1)\), defined for \(|\tau|<\tau_1\), such that
\[
\partial_r\phi_\tau(r_\tau)=0,
\qquad
r_\tau\longrightarrow r_0
\quad\text{as }\tau\to0.
\]

The two symmetries used above persist for \(\tau\neq0\). Thus
\[
\nabla\Phi_\tau(r_\tau,-r_\tau)
\in\operatorname{span}\{e_r\},
\]
while
\[
\nabla\Phi_\tau(r_\tau,-r_\tau)\cdot e_r
=
\phi_\tau'(r_\tau)=0.
\]
Hence
\[
\nabla\Phi_\tau(r_\tau,-r_\tau)=0.
\]
Moreover, the four directions \(e_r,e_c,e_v,e_\theta\) still
diagonalize the Hessian. By continuity, for all sufficiently small
\(\tau>0\),
\[
D^2\Phi_\tau(r_\tau,-r_\tau)[e,e]>0,
\qquad
e\in\{e_r,e_c,e_v\}.
\]
It remains to analyze the rotational direction. Set
\[
q_\tau:=3\tau r_\tau^2.
\]
Along the rotational path
\[
(z_1,z_2)
=
(r_\tau e^{i\theta},-r_\tau e^{i\theta}),
\]
the disk part of \(\Phi_\tau\) is rotation invariant. The conformal
perturbation contributes
\[
-2\ln|1+q_\tau e^{2i\theta}|
=
-\ln\bigl(1+2q_\tau\cos(2\theta)+q_\tau^2\bigr).
\]
Consequently,
\[
\left.
\frac{d^2}{d\theta^2}
\Phi_\tau
(r_\tau e^{i\theta},-r_\tau e^{i\theta})
\right|_{\theta=0}
=
\frac{8q_\tau}{(1+q_\tau)^2}>0.
\]
The tangent vector of this path at \(\theta=0\) is
\(r_\tau e_\theta\). Since
\[
D\Phi_\tau(r_\tau,-r_\tau)=0,
\]
the second-derivative chain rule gives
\[
\left.
\frac{d^2}{d\theta^2}
\Phi_\tau
(r_\tau e^{i\theta},-r_\tau e^{i\theta})
\right|_{\theta=0}
=
r_\tau^2
D^2\Phi_\tau(r_\tau,-r_\tau)[e_\theta,e_\theta].
\]
Therefore,
\[
D^2\Phi_\tau(r_\tau,-r_\tau)[e_\theta,e_\theta]>0.
\]
Together with the positivity of the other three directions, this shows
that
\[
D^2\Phi_\tau(r_\tau,-r_\tau)
\]
is positive definite for all sufficiently small \(\tau>0\).

Define
\[
\mathcal F_\tau(z_1,z_2)
:=
\bigl(F_\tau(z_1),F_\tau(z_2)\bigr),
\qquad
z_\tau:=(r_\tau,-r_\tau),
\qquad
a_\tau:=\mathcal F_\tau(z_\tau).
\]
Since
\[
\Phi_\tau
=
\frac{4\pi}{\mu^2}
W_{\boldsymbol\kappa}^{(2)}\circ\mathcal F_\tau
\]
and \(D\mathcal F_\tau(z_\tau)\) is invertible, the criticality of
\(z_\tau\) implies
\[
DW_{\boldsymbol\kappa}^{(2)}(a_\tau)=0.
\]
At this critical point, the term involving \(D^2\mathcal F_\tau\) in the
second-derivative chain rule vanishes because
\(DW_{\boldsymbol\kappa}^{(2)}(a_\tau)=0\). Therefore,
\[
D^2\Phi_\tau(z_\tau)
=
\frac{4\pi}{\mu^2}
[D\mathcal F_\tau(z_\tau)]^T
D^2W_{\boldsymbol\kappa}^{(2)}(a_\tau)
D\mathcal F_\tau(z_\tau).
\]
Since \(D^2\Phi_\tau(z_\tau)\) is positive definite and
\(D\mathcal F_\tau(z_\tau)\) is invertible, it follows that
\(D^2W_{\boldsymbol\kappa}^{(2)}(a_\tau)\) is positive definite. Therefore,
\[
\bigl(F_\tau(r_\tau),F_\tau(-r_\tau)\bigr)
\]
is a strict nondegenerate local minimum point of
\(W_{\boldsymbol\kappa}^{(2)}\).

Finally, since the normalized pullback \(\Phi_\tau\) is independent of
\(\mu\), we may choose \(\tau_0\in(0,1/3)\) sufficiently small such that
the conclusion holds for every \(\mu>0\).
\end{proof}

\section*{Statements and Declarations}

\noindent\textbf{Funding.}
D. Cao was supported by National Key R\&D Program (Grant 2023YFA1010001) and NNSF of China (Grant
12371212).  G. Wang was supported by NNSF of China (Grant 12471101).

\smallskip
\noindent\textbf{Conflicts of interest.}
The authors declare that they have no conflicts of interest related to this work.

\smallskip
\noindent\textbf{Data availability.}
Data sharing is not applicable to this article, as no datasets were generated or analyzed during the current study.

\end{document}